\documentclass[11pt]{amsart}
\usepackage[latin1]{inputenc}
\usepackage{amsmath,latexsym,amssymb,graphicx,dsfont}
\usepackage[square]{natbib}

\newcommand{\R}{\mathbb{R}}

\newtheorem{Th}{Theorem}[section]
\newtheorem{Lemma}[Th]{Lemma}
\newtheorem{Cor}[Th]{Corollary}
\newtheorem{Prop}[Th]{Proposition}

\newcommand{\E}{\mathbb{E}}
\newcommand{\Prob}{\mathbb{P}}

\newcommand{\calA}{\mathcal{A}}

\newcommand{\calI}{\mathcal{I}}

\newcommand{\N}{\mathbb{N}}
\newcommand{\Z}{\mathbb{Z}}
\newcommand{\eps}{\varepsilon}

\newcommand{\W}[1]{\mathbf{W}_{#1}}
\newcommand{\e}[1]{\mathbf{e}_{#1}}
\renewcommand{\t}{\mathbf{\tau}}
\renewcommand{\k}{\mathbf{k}}

\numberwithin{equation}{section}

\begin{document}

\title[Asymptotic Entropy of Random Walks on Free Products]{Asymptotic Entropy of Random Walks on Free Products}

\maketitle

\centerline{\scshape Lorenz A. Gilch\footnote{Research supported by German Research Foundation (DFG) grant GI 746/1-1.}}
\medskip
\begin{center}
{\footnotesize 
Graz University of Technology\\
Institut f\"ur mathematische Strukturtheorie (Math. C)\\
Steyrergasse 30\\
A-8010 Graz\\
Austria\\[1ex]
Email: gilch@TUGraz.at\\
URL: http://www.math.tugraz.at/$\sim$gilch/}
\end{center}

\begin{abstract}
Suppose we are given the free product V of a finite family of finite or
  countable sets. We
  consider a transient random walk on the free product arising naturally from a
  convex combination of random walks on the free factors. We prove the existence of the asymptotic
  entropy and present three  different, equivalent formulas, which
  are derived by three different techniques. In particular, we will show that
  the entropy is the rate of escape with respect to the Greenian
  metric. Moreover, we link asymptotic entropy with the rate of escape and
  volume growth resulting in two inequalities.
\end{abstract}
\vspace{1cm}
{\footnotesize 
Submitted to EJP on June 23, 2010; final version accepted on December 3, 2010.
}

\newpage

\section{Introduction}

Suppose we are given a finite family of finite or countable sets
$V_1,\dots,V_r$ with distinguished vertices $o_i\in V_i$ for $i\in\{1,\dots,r\}$. The free product of the sets $V_i$ is given by
$V:=V_1\ast \ldots\ast V_r$, the set of all finite words of the form $x_1\dots
x_n$ such that each letter is an element of $\bigcup_{i=1}^r
V_i\setminus\{o_i\}$ and two consecutive letters arise not from the same
$V_i$. We consider a transient Markov chain $(X_n)_{n\in\mathbb{N}_0}$ on $V$
starting at the empty word $o$,
which arises from a convex combination of transition probabilities on the sets
$V_i$. Denote by $\pi_n$ the distribution of $X_n$. We are interested in whether
the sequence  $\mathbb{E}[-\log \pi_n(X_n)]/n$ converges, and if so, to compute this constant. If the limit exists, it is called the
\textit{asymptotic entropy}. In this paper, we study this question for random
walks on general free products. In particular, we will derive three different
formulas for the entropy by using three different techniques.
\par
Let us outline some results about random walks on free products:
for free products of finite groups, Mairesse
and Math\'eus \cite{mairesse1} computed an explicit formula for the rate
of escape and asymptotic entropy by solving a finite system of polynomial
equations. Their result remains valid in the case of free products of infinite
groups, but one needs then to solve an infinite system of polynomial equations.
Gilch \cite{gilch:07} computed two different formulas for the rate
of escape with respect to the word length of random walks on free products of graphs by different techniques,
and also a third formula for free products of (not necessarily finite)
groups. The techniques of \cite{gilch:07} are adapted to the present setting. Asymptotic behaviour of return probabilities of random walks on free products has also been
studied in many ways; e.g. Gerl and Woess \cite{gerl-woess}, \cite{woess3},  
Sawyer \cite{sawyer78}, Cartwright and Soardi \cite{cartwright-soardi}, and
\mbox{Lalley \cite{lalley93},} Candellero and Gilch \cite{candellero-gilch}. 
\par
Our proof of existence of the entropy envolves generating functions techniques.
The techniques we use for rewriting probability generating functions in
terms of functions on the factors of the free product were introduced independently and
simultaneously by Cartwright and Soardi \cite{cartwright-soardi}, Woess
\cite{woess3}, Voiculescu \cite{voiculescu} and McLaughlin \cite{mclaughlin}.
In particular, we will see that asymptotic entropy
is the rate of escape with respect to a distance function in terms of Green
functions. While it
is well-known by Kingman's subadditive ergodic theorem (see Kingman \cite{kingman}) that entropy (introduced by Avez
\cite{avez72}) exists for random walks on groups whenever $\mathbb{E}[-\log
 \pi_1(X_1)]<\infty$, existence for
random walks on other structures is not known a priori. We are not able to
apply Kingman's theorem in our present setting, since
we have no (general) subadditivity and we have only a partial composition law
for two elements of the free product. For more details about
entropy of random walks on groups we refer to Kaimanovich and Vershik
\cite{kaimanovich-vershik} and Derriennic
\cite{derrienic}. 
\par
An important link
between drifts and harmonic analysis was obtained by \mbox{Varopoulos
\cite{varopoulos}.} He proved that for symmetric finite range random walks on groups the existence of non-trivial bounded
harmonic functions is equivalent to a non-zero rate of escape. Karlsson and
Ledrappier \cite{karlsson-ledrappier07} generalized this result to symmetric
random walks with finite first moment of the step lengths.
This leads to a
link between the rate of escape and the entropy of random walks, compare
e.g. with Kaimanovich and Vershik \cite{kaimanovich-vershik} and \mbox{Erschler
\cite{erschler2}.} Erschler and Kaimanovich \cite{erschler-kaimanovich:entropy-continuity} asked if drift and entropy of random walks on groups vary continuously on the probability measure, which governs the random walk. We prove real-analyticity of the entropy when varying the probabilty measure of constant support; compare also with the recent work of Ledrappier \cite{ledrappier:10}, who simultaneously proved this property for finite-range random walks on free groups.
\par
Apart from the proof of existence of the asymptotic entropy $h=\lim_{n\to\infty} \mathbb{E}[-\log \pi_n(X_n)]/n$ (Theorem \ref{thm:entropy}), we will calculate explicit formulas for the entropy (see Theorems \ref{thm:entropy}, \ref{th:entropy2}, \ref{thm:entropy-groups} and Corollary \ref{Cor:entropy-dgf}) and we will show that the entropy is non-zero. The technique of our proof of existence of the entropy was motivated by Benjamini and Peres \cite{benjamini-peres94}, where it
is shown that for random walks on groups the entropy equals the rate of escape w.r.t. the Greenian distance; compare also with
Blach\`ere, Ha\"issinsky and Mathieu \cite{blachere-haissinsky-mathieu}. We are also able to show that, for random walks on free products of graphs, the asymptotic entropy equals just the rate of escape w.r.t. the Greenian distance (see Corollary \ref{cor:greenRoE} in view of Theorem \ref{thm:entropy}). Moreover, we prove convergence in probability and convergence in $L_1$ (if the non-zero single transition probabilities are bounded away from $0$)
of the sequence $-\frac{1}{n}\log \pi_n(X_n)$ to $h$ (see Corollary \ref{cor:entropy3}), and we show also that $h$ can be computed along almost every sample path as the limes inferior of the aforementioned sequence (Corollary \ref{cor:entropy2}). In the case of random walks on discrete
groups, Kingman's subadditive ergodic theorem provides both the almost sure
convergence and the convergence in $L_1$ to the asymptotic entropy; in the case
of general free products there is neither a global composition law for elements
of the free product nor subadditivity. Thus, in the latter case we have to introduce and
investigate new processes. The question of almost sure convergence of
$-\frac{1}{n} \log \pi_n(X_n)$ to some constant $h$, however, remains open.
Similar results concerning existence and formulas for the entropy are proved in Gilch and
 M\"uller \cite{gilch-mueller} for random walks on directed
 covers of graphs. The reasoning of our proofs follows the argumentation in \cite{gilch-mueller}: we will show that the entropy equals the rate of escape w.r.t. some special length function, and we deduce the proposed properties analogously. In the present case of free products of graphs, the reasoning is getting more complicated due to the more complex structure of free products in contrast to directed covers, although the main results about existence and convergence types are very similar. We will point out these difficulties and main differences to \cite{gilch-mueller} at the end of Section \ref{subsec:entropy}. Finally, we will link entropy with the rate of escape and the growth rate of the free product, resulting in two inequalities \mbox{(Corollary \ref{cor:inequalities}).}
\par
The plan of the paper is as follows: in Section \ref{sec:preliminaries} we
define the random walk on the free product and the associated generating
functions. In Section \ref{sec:entropy} we prove existence of the asymptotic entropy and
give also an explicit formula for it. Another formula is derived in Section
\ref{sec:dgf} with the help of double generating functions and a theorem of Sawyer and
Steger \cite{sawyer}. In Section \ref{sec:groups} we use another technique to compute a third
explicit formula for the entropy of random walks on free products of (not
necessarily finite) groups. \mbox{Section \ref{sec:entropy-inequalities}} links
entropy with the rate of escape and the growth rate of the free product.
Sample computations are presented in Section \ref{sec:examples}.

\section{Random Walks on Free Products}

\label{sec:preliminaries}
\subsection{Free Products and Random Walks}
Let $\mathcal{I}:=\{1,\dots,r\}\subseteq \N$, where $r\geq 2$. For each
$i\in\mathcal{I}$, consider a random walk with
transition matrix $P_i$ on a finite
or countable state space $V_i$. W.l.o.g. we assume that the sets $V_i$ are
pairwise disjoint and we exclude the case $r=2=|V_1|=|V_2|$ (see below for
further explanation). The corresponding single and $n$-step transition
probabilities are denoted by $p_i(x,y)$ and $p_i^{(n)}(x,y)$, where
$x,y\in V_i$. For every $i\in\mathcal{I}$, we select an element $o_i$ of $V_i$
as the ``root''. To help visualize this, we
think of graphs $\mathcal{X}_i$ with vertex sets $V_i$ and roots $o_i$ such that there is
an oriented edge $x\to y$ if and only if $p_i(x,y)>0$. Thus, we have a natural
graph metric on the set $V_i$. Furthermore, we shall
assume that for every $i\in\mathcal{I}$ and every $x\in V_i$ there is some
$n_x\in\mathbb{N}$ such that $p^{(n_x)}_i(o_i,x)>0$. For sake of simplicity we
assume $p_i(x,x)=0$ for every $i\in\mathcal{I}$ and $x\in V_i$. Moreover, we assume that the random walks on $V_i$ are \textit{uniformly irreducible}, that is,
there are $\eps_0^{(i)}>0$ and $K_i\in\N$ such that for all $x,y\in V_i$
\begin{eqnarray}\label{def:uniform-irred}
p_i(x,y) > 0 \quad & \Rightarrow & \quad  p_i^{(k)}(x,y)\geq \eps_0^{(i)} \quad \textrm{ for
  some } k\leq K_i.
\end{eqnarray}
We set $K:=\max_{i\in\mathcal{I}} K_i$ and
$\varepsilon_0:=\min_{i\in\mathcal{I}} \varepsilon_0^{(i)}$.
For instance, this property is satisfied for nearest neighbour random walks on
Cayley graphs of finitely generated groups, which are governed by probability
measures on the groups.
\par
Let $V_i^\times := V_i\setminus\{o_i\}$ for every $i\in\mathcal{I}$ and let
$V_\ast^\times :=\bigcup_{i\in\mathcal{I}}V_i^\times$. The \textit{free product} is
given by
\begin{eqnarray}\label{freeproduct}
V& := & V_1\ast \ldots \ast V_r \nonumber\\
& = &  \Big\lbrace x_1x_2\dots x_n\ \Bigl|\
n\in\mathbb{N}, x_j\in V_\ast^\times, x_j\in V_k^\times \Rightarrow
x_{j+1}\notin V_k^\times\ \Big\rbrace\cup \Big\lbrace o \Big\rbrace. 
\end{eqnarray}
The elements of $V$ are ``words'' with letters, also called \textit{blocks},
  from the sets $V_i^\times$ such that no two consecutive letters come from the same
  $V_i$. The empty word $o$ describes the root \mbox{of $V$.} If $u=u_1\dots u_m\in V$ and $v=v_1\dots
v_n\in V$ with $u_m\in V_i$ and $v_1\notin V_i$ then $uv$ stands for their
concatenation as words. This is only a partial composition law, which makes
defining the asymptotic entropy more complicated than in the case of free
products of \textit{groups}. 
In particular, we set $uo_i:=u$ for all
$i\in\mathcal{I}$ and $ou:=u$. Note that $V_i\subseteq V$ and $o_i$ as a word in
$V$ is identified with $o$. The \textit{block length} of a word $u=u_1\dots u_m$
is given by $\Vert u\Vert:=m$. Additionally, we set $\Vert o\Vert:=0$. The
\textit{type} $\tau(u)$ of $u$ is defined to be $i$ if $u_m\in V_i^\times$; we
set $\tau(o):=0$. Finally, $\tilde u$ denotes the last letter $u_m$ of $u$. The
set $V$ can again be interpreted as the vertex set of a graph $\mathcal{X}$, which is
constructed as follows: take copies of
$\mathcal{X}_1,\dots \mathcal{X}_r$ and glue them together at their roots to one single common
root, which becomes $o$; inductively, at each vertex $v_1\dots v_k$ with
$v_k\in V_i$ attach a copy of every $\mathcal{X}_j$, $j\neq i$, and so on. Thus, we have
also a natural graph metric associated to the elements \mbox{in $V$.}
\par
The next step is the construction of a new Markov chain on the \textit{free
  product}. For this purpose, we lift $P_i$ to a transition matrix $\bar P_i$ on $V$: if
$x\in V$ with $\tau(x)\neq i$ and $v,w\in V_i$,
then $\bar p_i(xv,xw):=p_i(v,w)$. Otherwise we set $\bar p_i(x,y):=0$. We choose
$0<\alpha_1,\dots,\alpha_r\in\mathbb{R}$ with $\sum_{i\in\mathcal{I}} \alpha_i =
1$. Then we obtain a new transition matrix on $V$ given by
$$
P=\sum_{i\in\mathcal{I}} \alpha_i \bar P_i.
$$
The random walk on $V$ starting at $o$, which is governed by $P$, is described by the sequence of random variables
$(X_n)_{n\in\mathbb{N}_0}$. For $x,y\in V$, the associated single and $n$-step transition
probabilities are denoted by $p(x,y)$ and $p^{(n)}(x,y)$. Thus, $P$ governs a
nearest neighbour random walk on the graph $\mathcal{X}$, where $P$ arises from a convex
combination of the nearest neighbour random walks on the graphs $\mathcal{X}_i$.
\par
Theorem 3.3 in \cite{gilch:07} shows existence (including a formula) of a positive number $\ell_0$ such that
$\ell_0=\lim_{n\to\infty} \Vert X_n\Vert/n$ almost surely. The number $\ell_0$
is called the \textit{rate of escape w.r.t. the block length}. Denote by
$\pi_n$ the distribution of $X_n$. If there is a real number $h$ such that 
$$
h= \lim_{n\to\infty} \frac{1}{n}\mathbb{E}\bigl[-\log \pi_n(X_n)\bigr],
$$ 
then $h$ is called the \textit{asymptotic entropy} of the process
$(X_n)_{n\in\N_0}$; we write $\N_0:=\N\setminus \{0\}$. If the sets $V_i$ are groups and the random walks $P_i$ are
governed by probability measures $\mu_i$, existence of the asymptotic
entropy rate is well-known, and in this case we even have \mbox{$h= \lim_{n\to\infty}
-\frac{1}{n} \log \pi_n(X_n)$} almost surely; see Derriennic \cite{derrienic} and
Kaimanovich and \mbox{Vershik \cite{kaimanovich-vershik}.} We prove existence of $h$
in the case of general free products.

\subsection{Generating Functions}
Our main tool will be the usage of generating functions, which we introduce now.
The \textit{Green functions} related to $P_i$ and $P$ are given by
$$
G_i(x_i,y_i|z):= \sum_{n\geq 0} p_i^{(n)}(x_i,y_i)\,z^n \quad \textrm{ and } \quad 
G(x,y|z):= \sum_{n\geq 0} p^{(n)}(x,y)\,z^n,
$$
where $z\in\mathbb{C}$, $x_i,y_i\in V_i$ and $x,y\in V$. At this point we make
the \textit{basic assumption} that the radius of convergence $R$ of
$G(\cdot,\cdot|z)$ is strictly bigger than $1$. This implies
\textit{transience} of our random walk on $V$. Thus, we may exclude the case
$r=2=|V_1|=|V_2|$, because we get recurrence in this case. For instance, if all $P_i$ govern \textit{reversible} Markov chains, then $R>1$; see \cite[Theorem 10.3]{woess}. Furthermore, it is easy to see that $R>1$ holds also if there is some $i\in\mathcal{I}$ such that $p_i^{(n)}(o_i,o_i)=0$ for all $n\in\N$.
\par
The \textit{first visit
  generating functions} related to $P_i$ and $P$ are given by
\begin{eqnarray*}
F_i(x_i,y_i|z) & := & \sum_{n\geq 0} \Prob\bigl[Y_n^{(i)}=y_i,\forall m\leq n-1:
Y_m^{(i)}\neq y_i \mid Y_0^{(i)}=x_i\bigr]z^n\ \textrm{ and}\\
F(x,y|z) & := & \sum_{n\geq 0} \Prob\bigl[X_n=y,\forall m\leq n-1:
X_m\neq y \mid X_0=x\bigr]z^n,
\end{eqnarray*}
where $\bigl(Y_n^{(i)}\bigr)_{n\in\mathbb{N}_0}$ describes a random walk on
$V_i$ governed by $P_i$.
The stopping time of the first return to $o$ is defined as \mbox{$T_o:=\inf\{m\geq 1 \mid X_m=o \}$.}
For $i\in\mathcal{I}$, define
$$
\overline{H}_i(z):= \sum_{n\geq 1} \Prob[T_o=n,X_1\notin V_i^\times]\,z^n \
\textrm{ and } \
\xi_i(z) := \frac{\alpha_i z}{1-\overline{H}_i(z)}.
$$
We write also $\xi_i:=\xi_i(1)$, $\xi_{\min}:=\min_{i\in\mathcal{I}} \xi_i$ and
$\xi_{\max}:=\max_{i\in\mathcal{I}} \xi_i$. Observe that $\xi_i<1$; see 
\cite[Lemma 2.3]{gilch:07}.
We have $F(x_i,y_i|z)=F_i\bigl(x_i,y_i|\xi_i(z)\bigr)$ for all $x_i,y_i\in V_i$; see Woess
\cite[Prop. 9.18c]{woess}. Thus,
$$
\xi_i(z) := \frac{\alpha_i z}{1-\sum_{j\in\mathcal{I}\setminus\{i\}}
  \sum_{s\in V_j} \alpha_j p_j(o_j,s)z F_j\bigl(s,o_j \bigl| \xi_j(z)\bigr)}.
$$
For $x_i\in V_i$ and $x\in V$, define the stopping times $T^{(i)}_{x_i}:=\inf\{m\geq 1 \mid Y_m^{(i)}=x_i\}$ and
$T_{x}:=\inf\{m\geq 1 \mid X_m=x\}$, which take both values in $\N\cup\{\infty \}$.
Then the \textit{last visit generating functions} related to $P_i$ and $P$ are defined as
\begin{eqnarray*}
L_i(x_i,y_i|z) & := &\sum_{n\geq 0} \Prob\bigl[Y_n^{(i)}=y_i, T^{(i)}_{x_i}>n \mid Y_0^{(i)}=x_i\bigr]\,z^n,\\
L(x,y|z) & := &\sum_{n\geq 0} \Prob\bigl[X_n=y, T_x>n \mid X_0=x\bigr]\,z^n.
\end{eqnarray*}
If $x=x_1\dots x_n,y=x_1\dots x_{n}x_{n+1}\in V$ with $\tau(x_{n+1})=i$ then
\begin{equation}\label{l-equation}
L(x,y|z) = L_i\bigl(o_i,x_{n+1}\,\bigr|\,\xi_i(z)\bigr);
\end{equation}
this equation is proved completely analogously to \cite[Prop. 9.18c]{woess}. If all paths from $x\in V$ to $w\in V$ have
to pass through $y\in V$, then
$$
L(x,w|z) = L(x,y|z) \cdot L(y,w|z);
$$
this can be easily checked by conditioning on the last visit of $y$ when
walking from $x$ \mbox{to $w$.}
We have the following important equations, which follow by conditioning on the
last visits of $x_i$ and $x$, the first visits of $y_i$ and $y$ respectively:
\begin{equation}\label{gl-equations}
\begin{array}{rcl}
G_i(x_i,y_i|z) & = & G_i(x_i,x_i|z)\cdot L_i(x_i,y_i|z) = F_i(x_i,y_i|z)\cdot G_i(y_i,y_i|z),\\[1ex]
G(x,y|z) & = & G(x,x|z)\cdot L(x,y|z) = F(x,y|z)\cdot G(y,y|z).
\end{array}
\end{equation}
Observe that the generating functions $F(\cdot,\cdot|z)$ and
$L(\cdot,\cdot|z)$ 
have also radii of convergence strictl bigger than $1$.

\section{The Asymptotic Entropy}

\label{sec:entropy}

\subsection{Rate of Escape w.r.t. specific Length Function}

In this subsection we prove existence of the rate of escape with respect to a
specific length function. From this we will deduce existence and a formula for
the asymptotic entropy in the upcoming subsection. 
\par
We assign to each element $x_i\in V_i$ the ``length''
$$
l_i(x_i):=-\log L(o,x_i|1)=-\log L_i(o_i,x_i|\xi_i).
$$
We extend it to a length
function on $V$ by assigning to $v_1\dots v_n\in V$ the length
$$
l(v_1\dots v_n) := \sum_{i=1}^n l_{\tau(v_i)}(v_i)=-\sum_{i=1}^n \log L(o,v_i|1) = -\log L(o,v_1\dots v_n|1).
$$
Observe that the lengths can also be negative. E.g., this can be interpreted as
height differences. The aim of this subsection is to show existence of a number
$\ell\in\mathbb{R}$ such that the quotient $l(X_n)/n$ tends to
$\ell$ almost surely as $n\to\infty$. We call $\ell$ the \textit{rate of escape
  w.r.t. the length function $l(\cdot)$}.
\par
We follow now the reasoning of \cite[Section 3]{gilch:07}. Denote by $X_n^{(k)}$
the projection of $X_n$ to the first $k$ letters. We define the
\textit{$k$-th exit time} as
$$
\mathbf{e}_k := \min\bigl\lbrace m\in\N_0 \ \bigl| \ \forall n\geq m: X_n^{(k)}
\textrm{ is constant}\bigr\rbrace.
$$
Moreover, we define $\mathbf{W}_k:=X_{\mathbf{e}_k}$,
$\tau_k:=\tau(\mathbf{W}_k)$ and $\mathbf{k}(n):=\max \{k\in\N_0 \mid
\mathbf{e}_k\leq n\}$. 
We remark that $\Vert X_n\Vert \to\infty$ as
$n\to\infty$, and consequently $\mathbf{e}_k<\infty$ almost surely for every
$k\in\N$; see \cite[Prop. 2.5]{gilch:07}. 
Recall that $\widetilde{\mathbf{W}}_k$ is just the laster letter of the random word $X_{\mathbf{e}_k}$.
The process $(\tau_k)_{k\in\N}$ is Markovian and has transition probabilities
$$
\hat q(i,j) = \frac{\alpha_j}{\alpha_i} \frac{\xi_i}{\xi_j}\frac{1-\xi_j}{1-\xi_i} \Bigl(\frac{1}{(1-\xi_j)G_j(o_j,o_j|\xi_j)}-1\Bigr)
$$
for $i\neq j$ and $\hat q(i,i)=0$; see \cite[Lemma 3.4]{gilch:07}. This process is positive recurrent with invariant probability measure
\begin{eqnarray*}
\nu(i) & = & C^{-1} \cdot \frac{\alpha_i (1-\xi_i)}{\xi_i} \bigl(1-(1-\xi_i)
G_i(o_i,o_i|\xi_i)\bigr), \\
\textrm{ where } C & := &\sum_{i\in\calI} \frac{\alpha_i (1-\xi_i)}{\xi_i} \bigl(1-(1-\xi_i)
G_i(o_i,o_i|\xi_i)\bigr);
\end{eqnarray*}
see \cite[Section 3]{gilch:07}. Furthermore, the rate of escape w.r.t.
the block length exists almost surely and is given by the almost sure constant limit
$$
\ell_0=\lim_{n\to\infty}\frac{\Vert X_n\Vert}{n} = \lim_{k\to\infty}
\frac{k}{\mathbf{e}_k} = \frac{1}{\sum_{i,j\in\mathcal{I}, i\neq j} \nu(i)\,\alpha_j \frac{1-\xi_j}{1-\xi_i} \gamma_{i,j}'(1)}
$$
(see \cite[Theorem 3.3]{gilch:07}), where 
$$
\gamma_{i,j}(z) := \frac{1}{\alpha_i} \frac{\xi_i(z)}{\xi_j(z)} \Bigl(
\frac{1}{\bigl(1-\xi_j(z)\bigr) G_j\bigl(o_j,o_j\big|\xi_j(z)\bigr)}-1\Bigr).
$$
\begin{Lemma}\label{lemma:tildeW-process}
The process $\bigl(\widetilde{\mathbf{W}}_k,\t_k\bigr)_{k\in\N}$ is Markovian and has
transition probabilities
$$
q\bigl((g,i),(h,j)\bigr) = 
\begin{cases}
\frac{\alpha_j}{\alpha_i}\frac{\xi_i}{\xi_j}\frac{1-\xi_j}{1-\xi_i} L_j(o_j,h|\xi_j), & \textrm{if }
i\neq j,\\
0, & \textrm{if } i=j.
\end{cases}
$$
Furthermore, the process is positive recurrent with invariant probability measure
$$
\pi (g,i) = \sum_{j\in\calI} \nu(j) q\bigl((\ast,j),(g,i)\bigr).
$$
\end{Lemma}
\textit{Remark:} Observe that the transition probabilities $q\bigl((g,i),(h,j)\bigr)$ of
$\bigl(\widetilde{\mathbf{W}}_k,\t_k\bigr)_{k\in\N}$ do not depend on
$g$. Therefore, we will write sometimes an asterisk instead of $g$.
\begin{proof}
By \cite[Section 3]{gilch:07}, the process
$\bigl(\widetilde{\mathbf{W}}_k,\e{k}-\e{k-1},\tau_k\bigr)_{k\in\N}$ is Markovian and
has transition probabilities
$$
\tilde q\bigl((g,m,i),(h,n,j)\bigr) = 
\begin{cases}
\frac{1-\xi_j}{1-\xi_i} \sum_{s\in V_j} k_i^{(n-1)}(s) p(s,h), & \textrm{if
} i\neq j,\\
0, & \textrm{if } i=j,
\end{cases}
$$
where $k_i^{(n)}(s):=\Prob\bigl[X_n=s, \forall l\leq n: X_l\notin
V_i^\times |X_0=o]$ for $s\in V^\times_\ast\setminus V_i$. Thus, $\bigl(\widetilde{\W{k}},\t_k\bigr)_{k\in\N}$
is also Markovian and has the following transition probabilities if $i\neq j$:
\begin{eqnarray*}
q\bigl((g,i),(h,j)\bigr) & = & \sum_{n\geq 1} \tilde
q\bigl((g,\ast,i),(h,n,j)\bigr) 
= \frac{1-\xi_j}{1-\xi_i} \sum_{s\in V_j} \sum_{n\geq 1} k_i^{(n-1)}(s) p(s,h)\\
&=& \frac{1-\xi_j}{1-\xi_i} \sum_{s\in
  V_j} \frac{L_j(o_j,s|\xi_j)}{1-\bar{H}_i(1)} p(s,h)
= \frac{\alpha_j}{\alpha_i} \frac{\xi_i}{\xi_j} \frac{1-\xi_j}{1-\xi_i} L_j(o_j,h|\xi_j).
\end{eqnarray*}
In the third equality we conditioned on the last visit of $o$ before finally
walking from $o$ to $s$ and we remark that $h\in V_j^\times$.
A straight-forward computation shows that $\pi$ is the invariant probability
measure of $\bigl(\widetilde{\mathbf{W}}_k,\t_k\bigr)_{k\in\N}$, where we write
$\mathcal{A}:=\bigl\lbrace (g,i)\,\bigl|\,i\in\calI, g\in V_i^\times\bigr\rbrace$:
\begin{eqnarray*}
\sum_{(g,i)\in\calA} \pi(g,i) \cdot q\bigl((g,i),(h,j)\bigr)
&=& \sum_{(g,i)\in\calA} \sum_{k\in\calI} \nu(k) \cdot
q\bigl((\ast,k),(g,i)\bigr) \cdot q\bigl((\ast,i),(h,j)\bigr)\\
&=& \sum_{i\in\calI} q\bigl((\ast,i),(h,j)\bigr) \sum_{k\in\calI} \nu(k) 
\sum_{g\in V_i^\times} q\bigl((\ast,k),(g,i)\bigr)\\
&=& \sum_{i\in\calI} q\bigl((\ast,i),(h,j)\bigr) \sum_{k\in\calI} \nu(k) \cdot \hat
q(k,i) \\
& = & \sum_{i\in\calI} q\bigl((\ast,i),(h,j)\bigr)\cdot \nu(i) 
= \pi(h,j).
\end{eqnarray*}
\end{proof}
Now we are able to prove the following:
\begin{Prop}\label{prop:ell-existence}
There is a number $\ell\in\R$ such that 
$$
\ell = \lim_{n\to\infty} \frac{l(X_n)}{n} \quad \textrm{almost surely}.
$$
\end{Prop}
\begin{proof}
Define $h:\mathcal{A}\to \R$ by $h(g,j):=l(g)$. Then $\sum_{\lambda=1}^k h\bigl(\widetilde{\mathbf{W}}_{\lambda},\t_{\lambda}\bigr) = \sum_{\lambda=1}^k l\bigl(\widetilde{\mathbf{W}}_{\lambda}\bigr) = l(\mathbf{W}_k)$.
An application of the ergodic theorem for positive recurrent Markov chains yields
$$
\frac{l(\mathbf{W}_k)}{k}=\frac{1}{k}\sum_{\lambda=1}^k h\bigl(\widetilde{\mathbf{W}}_{\lambda},\t_{\lambda}\bigr) 
\xrightarrow{n\to\infty} C_h := \int h \,d\pi,
$$
if the integral on the right hand side exists. We now show that this property holds. Observe that the values $G_j(o_j,g|\xi_j)$
 are uniformly bounded from above for all $(g,j)\in\mathcal{A}$:
$$
G_j(o_j,g|\xi_j) = \sum_{n\geq 0} p_j^{(n)}(o_j,g)\,\xi_j^n \leq \frac{1}{1-\xi_j} \leq \frac{1}{1-\xi_{\max}}.
$$
For $g\in V_\ast^\times$, denote by $|g|$ the smallest $n\in\N$ such that $p_{\tau(g)}^{(n)}(o_{\tau(g)},g)>0$.
Uniform irreducibility of the random walk $P_i$ on $V_i$ implies that
there are some $\eps_0>0$ and $K\in\N$ such that for all $j\in\mathcal{I}$,
$x_j,y_j\in V_j$ with $p_j(x_j,y_j)>0$ we have
$p_j^{(k)}(x_j,y_j)\geq \eps_0$ for some $k\leq K$. Thus, for
$(g,j)\in\mathcal{A}$ we have 
$$
G_j(o_j,g|\xi_j) \geq \eps_0^{|g|} \xi_j^{|g|\cdot K} \geq  \bigl(\eps_0\, \xi_{\min}^K\bigr)^{|g|}.
$$
Observe that the inequality $|g|\cdot \bigl|\log\bigl(\eps_0\, \xi_{\min}^K\bigr)\bigr| <
\log 1/(1-\xi_{\max})$ holds if and only if $|g| < \log
(1-\xi_{\max})/\log (\eps_0\,\xi_{\min}^K)$. Define the sets
$$
M_1 :=  \Bigl\lbrace g\in V_\ast^\times \, \Bigl|\,
|g| \geq \frac{\log (1-\xi_{\max})}{\log (\eps_0\,\xi_{\min}^K)} \Bigr\rbrace,\quad
M_2  :=  \Bigl\lbrace g\in V_\ast^\times \, \Bigl|\, 
|g| < \frac{\log (1-\xi_{\max})}{\log (\eps_0\,\xi_{\min}^K)} \Bigr\rbrace.
$$
Recall Equation (\ref{gl-equations}). We can now prove existence of $\int h \,d\pi$:
\begin{eqnarray*}
\int |h| \,d\pi & = & \sum_{(g,j)\in\calA} \bigl|\log L_j(o_j,g|\xi_j)\bigr|
\cdot \pi(g,j)
\\
&\leq &  \sum_{(g,j)\in\calA} \bigl|\log G_j(o_j,g|\xi_j)\bigr|\cdot \pi(g,j) +
\sum_{(g,j)\in\calA} \bigl|\log G_j(o_j,o_j|\xi_j)\bigr|\cdot \pi(g,j)\\
&\leq & \sum_{(g,j)\in\calA : g\in M_1} \bigl|\log G_j(o_j,g|\xi_j)\bigr|
\cdot\pi(g,j) \\
&& \quad 
+\sum_{(g,j)\in\calA : g\in M_2} \bigl|\log G_j(o_j,g|\xi_j)\bigr|\cdot \pi(g,j)+
\max_{j\in\mathcal{I}} \log
G_j(o_j,o_j|\xi_j)\\
&\leq & \sum_{(g,j)\in\calA : g\in M_1} \bigl|\log (\eps_0\xi_{\min}^K)^{|g|}|
\cdot\pi(g,j) \\
&&\quad + \sum_{(g,j)\in\calA : g\in M_2} \bigl|\log (1-\xi_{\max})\bigr|
\cdot\pi(g,j)+ \max_{j\in\mathcal{I}} \log
G_j(o_j,o_j|\xi_j)\\
&\leq & \sum_{(g,j)\in\calA : g\in M_1} \bigl|\log (\eps_0\xi_{\min}^K)|\cdot |g|\cdot
\pi(g,j)  \\
&& \quad + \bigl|\log  (1-\xi_{\max})\bigr| +\max_{j\in\mathcal{I}} \log
G_j(o_j,o_j|\xi_j) < \infty,
\end{eqnarray*}
since $\sum_{(g,j)\in\calA} |g|\cdot \pi(g,j)<\infty$; see \cite[Proof of Prop. 3.2]{gilch:07}.
From this follows that $l(\W{k})/k$ tends to $C_h$ almost surely. The next step is to show that 
\begin{equation}
\frac{l(X_n)-l(\W{\k(n)})}{n} \xrightarrow{n\to\infty} 0 \quad \textrm{ almost surely.}
\end{equation}
To prove this, assume now that we have the representations
$\W{\k(n)}=g_1g_2\dots g_{\k(n)}$ and $X_n=g_1g_2\dots g_{\k(n)}\dots
g_{\|X_n\|}$. Define $M:= \max\bigl\lbrace |\log (\eps_0\,\xi_{\min}^K)|,
|\log(1-\xi_{\max})|\bigr\rbrace$. Then:
\begin{eqnarray*}
\bigl|l(X_n)-l(\W{\k(n)}) \bigr| 
&=&\biggl| -\sum_{i=\k(n)+1}^{\|X_n \|} \log L_{\tau(g_i)}\bigl(o_{\tau(g_i)},g_i \mid \xi_{\tau(g_i)}\bigr)\biggr|\\
&\leq & \sum_{i=\k(n)+1}^{\|X_n \|} \bigg| \log
\frac{G_{\tau(g_i)}\bigl(o_{\tau(g_i)},g_i \mid
  \xi_{\tau(g_i)}\bigr)}{G_{\tau(g_i)}\bigl(o_{\tau(g_i)},o_{\tau(g_i)} \mid
  \xi_{\tau(g_i)}\bigr)}\biggr|\\
& \leq & \sum_{i=\k(n)+1 : g_i\in M_1}^{\|X_n \|} \big| \log
G_{\tau(g_i)}\bigl(o_{\tau(g_i)},g_i \mid \xi_{\tau(g_i)}\bigr)\bigr| \\
&&\quad + \sum_{i=\k(n)+1 : g_i\in M_2}^{\|X_n\|} \big| \log
G_{\tau(g_i)}\bigl(o_{\tau(g_i)},g_i \mid \xi_{\tau(g_i)}\bigr)\bigr| \\[1ex]
&&\quad + \bigl(\|X_n\|-\k(n)\bigr) \cdot  \big| \log (1-\xi_{\max})\bigr|
\end{eqnarray*}
\begin{eqnarray*}
&\leq & \sum_{i=\k(n)+1 : g_i\in M_1}^{\|X_n\|} \big| \log  (\eps_0\,\xi_{\min}^K)^{|g_i|}
\bigr| \\
&& \quad + \sum_{i=\k(n)+1 : g_i\in M_2}^{\|X_n\|} \big| \log
(1-\xi_{\max})\bigr| + \bigl(\|X_n\|-\k(n)\bigr) \cdot  \big| \log (1-\xi_{\max})\bigr|\\
&\leq & \sum_{i=\k(n)+1 : g_i\in M_1}^{\|X_n \|} |g_i|\cdot M  + \sum_{i=\k(n)+1 : g_i\in M_2}^{\|X_n\|}
 M+\bigl(\|X_n\|-\k(n)\bigr) \cdot M \\[1ex]
& \leq & 3 \cdot M\cdot (n-\e{\k(n)}).
\end{eqnarray*}
Dividing the last inequality by $n$ and letting $n\to\infty$ provides
analogously to Nagnibeda and Woess \cite[Section 5]{woess2} that
$\lim_{n\to\infty} \bigl(l(X_n)-l(\W{\k(n)})\bigr)/n=0$ almost surely. Recall also that
$k/\e{k}\to \ell_0$ and $\e{\k(n)}/n\to 1$ almost surely; compare \cite[Proof
of Theorem D]{woess2} and \cite[Prop. 3.2, Thm. 3.3]{gilch:07}. Now we can conclude:
\begin{equation}\label{ell0-limit}
\frac{l(X_n)}{n} = \frac{l(X_n)-l(\W{\k(n)})}{n}+\frac{l(\W{\k(n)})}{\k(n)}\frac{\k(n)}{\e{k(n)}}\frac{\e{\k(n)}}{n}
\xrightarrow{n\to\infty} C_h\cdot \ell_0 \quad \textrm{almost surely}.
\end{equation}
\end{proof}
We now compute the constant $C_h$ from the last proposition explicitly:
\begin{eqnarray}
 C_h & = & \sum_{(g,j)\in\calA} l(g) \cdot  \sum_{i\in\calI} \nu(i) \cdot q\bigl((\ast,i),(g,j)\bigr)\nonumber\\
&=& \sum_{\substack{i,j\in \mathcal{I},\\ i\neq j}}\sum_{g\in V_j^\times} -\log
L_j(o_j,g|\xi_j)\,\nu(i) \frac{\alpha_j}{\alpha_i}\frac{\xi_i}{\xi_j}
\frac{1-\xi_j}{1-\xi_i}\,L_j(o_j,g|\xi_j).\label{Ch-formel} 
\end{eqnarray}
We conclude this subsection with the following observation:
\begin{Cor}\label{cor:greenRoE}
The rate of escape $\ell$ is non-negative and it is the rate of escape w.r.t. the
Greenian metric, which is given by $d_\textrm{Green}(x,y):=-\log F(x,y|1)$. That is,
$$
\ell = \lim_{n\to\infty} -\frac{1}{n}\log F(e,X_n|1) \geq 0.
$$
\end{Cor}
\begin{proof}
By (\ref{gl-equations}), we get
$$
\ell  =  \lim_{n\to\infty} -\frac{1}{n} \log F(e,X_n|1) - \frac{1}{n} \log
G(X_n,X_n|1) + \frac{1}{n} \log G(o,o|1).
$$
Since $F(e,X_n|1)\leq 1$ it remains to show that $G(x,x|1)$ is uniformly
bounded in $x\in V$: for $v,w\in V$, the \textit{first visit generating function} is defined as 
\begin{equation}\label{def:1st-visit}
U(v,w|z)=\sum_{n\geq 1} \Prob\bigl[X_n=w,\forall m\in\{1,\dots,n-1\}: X_m\neq w
\mid X_0=v\bigr]\,z^n.
\end{equation}
Therefore, 
$$
G(x,x|z)=\sum_{n\geq 0} U(x,x|z)^n = \frac{1}{1-U(x,x|z)}.
$$
Since $U(x,x|z)<1$ for all $z\in[1,R)$, $U(x,x|0)=0$ and $U(x,x|z)$ is
continuous, stricly increasing and strictly convex, we must have $U(x,x|1)\leq
\frac{1}{R}$, that is, $1\leq G(x,x|1)\leq
\bigl(1-\frac{1}{R}\bigr)^{-1}$. This finishes the proof.
\end{proof}

\subsection{Asymptotic Entropy}
\label{subsec:entropy}
In this subsection we will prove that $\ell$ equals the asymptotic
entropy, and we will give explicit formulas for it. The technique of the proof which we will give was motivated by
Benjamini and Peres \cite{benjamini-peres94}, where it is shown that the
asymptotic entropy of random walks on discrete
groups equals the rate of escape w.r.t. the Greenian distance. The proof
follows the same reasoning as in Gilch and M\"uller \cite{gilch-mueller}. 
\par
Recall that we made the assumption that the spectral radius of
$(X_n)_{n\in\N_0}$ is strictly smaller than $1$, that is, the Green function
$G(o,o|z)$ has radius of convergence $R>1$. Moreover, the functions $\xi_i(z)$, $i\in\calI$,
have radius of convergence bigger than $1$. Recall that $\xi_i=\xi_i(1)<1$ for every $i\in\calI$. Thus, we can choose $\varrho\in
(1,R)$ such that $\xi_i(\varrho)<1$ for all $i\in\calI$.
We now need the
following three technical lemmas:
\begin{Lemma}\label{lemma:prob-bound}
For all $m,n\in\N_0$,
$$
p^{(m)}(o,X_n) \leq G(o,o|\varrho)\cdot  \Bigl(\frac{1}{1-\max_{i\in\calI} \xi_i(\varrho)}\Bigr)^n \cdot \varrho^{-m}.
$$
\end{Lemma}
\begin{proof}
Denote by $\mathcal{C}_\varrho$ the circle with radius $\varrho$ in the complex plane centered at
$0$. A straightforward computation shows  for $m\in\mathbb{N}_0$: 
$$
\frac{1}{2\pi i} \oint_{\mathcal{C}_\varrho} z^m \frac{dz}{z} = 
\begin{cases}
1, & \textrm{if } m=0,\\
0, & \textrm{if } m\neq 0.
\end{cases}
$$
Let be $x=x_1\dots x_t\in V$. An application of Fubini's Theorem yields
\begin{eqnarray*}
\frac{1}{2\pi i} \oint_{\mathcal{C}_\varrho} G(o,x|z)\,z^{-m} \frac{dz}{z} & = &
\frac{1}{2\pi i} \oint_{\mathcal{C}_\varrho} \sum_{k\geq 0} p^{(k)}(o,x) z^k\,z^{-m}
\frac{dz}{z}\\
&=&
\frac{1}{2\pi i} \sum_{k\geq 0} p^{(k)}(o,x) \oint_{\mathcal{C}_\varrho} z^{k-m}
\frac{dz}{z} = p^{(m)}(o,x).
\end{eqnarray*}
Since $G(o,x|z)$ is analytic on $\mathcal{C}_\varrho$, we have $|G(o,x|z)|\leq G(o,x|\varrho)$ for all
$|z|=\varrho$. Thus, 
$$
p^{(m)}(o,x) \leq \frac{1}{2\pi}\cdot \varrho^{-m-1}\cdot  G(o,x|\varrho)\cdot  2\pi \varrho = G(o,x|\varrho)\cdot  \varrho^{-m}.
$$
Iterated applications of equations
(\ref{l-equation}) and (\ref{gl-equations}) provide
$$
G(o,x|\varrho) =  G(o,o|\varrho) \,\prod_{k=1}^{\Vert x\Vert} L_{\tau(x_k)}\bigl(o_{\tau(x_k)}, x_k|\xi_i(\varrho)\bigr)
\leq  G(o,o|\varrho) \Bigl( \frac{1}{1-\max_{i\in\calI}
  \xi_i(\varrho)}\Bigr)^{\Vert x \Vert}.
$$
Since $\Vert X_n \Vert\leq n$, we obtain
$$
p^{(m)}(e,X_n) \leq G(o,o|\varrho)\cdot  \Bigl(\frac{1}{1-\max_{i\in\calI} \xi_i(\varrho)}\Bigr)^n \cdot \varrho^{-m}.
$$
\end{proof}
\begin{Lemma}\label{cut-lemma}
Let $(A_n)_{n\in\N}$, $(a_n)_{n\in\N}$, $(b_n)_{n\in\N}$ be sequences of
strictly positive numbers with $A_n=a_n+b_n$. Assume that $\lim_{n\to\infty}
-\frac{1}{n}\log A_n=c \in [0,\infty)$ and that $\lim_{n\to\infty} b_n/q^n
= 0$ for all $q\in(0,1)$. Then $\lim_{n\to\infty} -\frac{1}{n}\log a_n=c$. 
\end{Lemma}
\begin{proof}
Under the made assumptions it can not be that $\liminf_{n\to\infty} a_n/q^n
= 0$ for every \mbox{$q\in(0,1)$.} Indeed, assume that this would hold. Choose any $q>0$. 
Then  there is a subseqence $(a_{n_k})_{k\in\N}$ with $a_{n_k}/q^{n_k}\to
0$. Moreover, there is $N_q\in\N$ such that $a_{n_k},b_{n_k}<q^{n_k}/2$ for all $k\geq N_q$. But this implies
$$
-\frac{1}{n_k}\log(a_{n_k}+b_{n_k}) \geq -\frac{1}{n_k}\log(q^{n_k}) = -\log q.
$$
The last inequality holds for every $q>0$, yielding that $\limsup_{n\to\infty}
-\frac{1}{n}\log A_n=\infty$, a contradiction.
\par
Thus, there is some $N\in\N$ such that $b_n<a_n$ for all $n\geq N$. We get for
all $n\geq N$:
\begin{eqnarray*}
-\frac{1}{n}\log(a_n+b_n) &\leq & -\frac{1}{n}\log(a_n) = 
-\frac{1}{n}\log\Bigl(\frac{1}{2}a_n+\frac{1}{2}a_n\Bigr) \\
&\leq &
-\frac{1}{n}\log\Bigl(\frac{1}{2}a_n+\frac{1}{2}b_n\Bigr) \leq 
 -\frac{1}{n}\log \frac{1}{2} -\frac{1}{n}\log(a_n+b_n).
\end{eqnarray*}
Taking limits yields that $-\frac{1}{n}\log(a_n)$ tends to $c$, since the
leftmost and rightmost side of this inequality chain tend to $c$.
\end{proof}
For the next lemma recall the definition of $K$ from $(\ref{def:uniform-irred})$.
\begin{Lemma}\label{lemma:sum-bounds}
For $n\in\N$, consider the function $f_n:V\to\R$ defined by
$$
f_n(x):=\begin{cases}
-\frac{1}{n}\log \sum_{m=0}^{Kn^2} p^{(m)}(o,x), & \textrm{if } p^{(n)}(o,x)>0,\\
0, & \textrm{otherwise.}
\end{cases}
$$
Then there are constants $d$ and $D$ such that $d\leq f_n(x)\leq D$ for all
$n\in\N$ and $x\in V$.
\end{Lemma}
\begin{proof}
Assume that $p^{(n)}(o,x)>0$.
Recall from the proof of Corollary \ref{cor:greenRoE} that we have $G(x,x|1)\leq
\bigl(1-\frac{1}{R}\bigr)^{-1}$. Therefore,
$$
\sum_{m=0}^{Kn^2} p^{(m)}(o,x) \leq G(o,x|1) \leq F(o,x|1)\cdot G(x,x|1)  \leq
\frac{1}{1-\frac{1}{R}}, 
$$
that is
$$
f_n(x) \geq -\frac{1}{n} \log \frac{1}{1-\frac{1}{R}}.
$$
For the upper bound, observe that, by uniform irreducibility, $x\in V$ with
$p^{(n)}(o,x)>0$ can be reached from $o$ in
$N_x\leq K\cdot |x| \leq Kn$ steps with a probability of at least $\varepsilon_0^{|x|}$,
where $\varepsilon_0 >0$ from (\ref{def:uniform-irred}) is independent from $x$. Thus, at least one of the summands in $\sum_{m=0}^{Kn^2} p^{(m)}(o,x)$ has a value
greater or equal to $\varepsilon_0^{|x|}\geq \varepsilon_0^{n}$. Thus,
$f_n(x)\leq -\log\varepsilon_0$.
\end{proof}
Now we can state and prove our first main result:
\begin{Th}\label{thm:entropy}
Assume $R>1$. Then the asymptotic entropy exists and is given by
$$
h= \ell_0 \cdot \sum_{g\in V_\ast^\times} l(g)\, \pi\bigl(g,\tau(g)\bigr) =\ell.
$$
\end{Th}
\begin{proof}
By (\ref{gl-equations}) we can rewrite $\ell$ as
$$
\ell=\lim_{n\to\infty} -\frac{1}{n} \log L(o,X_n|1) = \lim_{n\to\infty} -\frac{1}{n}
\log \frac{G(o,X_n|1)}{G(o,o|1)} = \lim_{n\to\infty} -\frac{1}{n} \log G(o,X_n|1).
$$
Since 
$$
G(o,X_n|1) = \sum_{m\geq 0} p^{(m)}(o,X_n) \geq p^{(n)}(o,X_n) = \pi_n(X_n),
$$
we have
\begin{equation}\label{equ:liminf-h}
\liminf_{n\to\infty} -\frac{1}{n} \log \pi_n(X_n) \geq \ell.
\end{equation}
The next aim is to prove $\limsup_{n\to\infty} -\frac{1}{n}\E\bigl[\log
\pi_n(X_n)\bigr] \leq \ell$. 
We now apply Lemma \ref{cut-lemma} by setting 
$$
A_n:=\sum_{m\geq 0} p^{(m)}(o,X_n),\ 
a_n:=\sum_{m=0}^{Kn^2} p^{(m)}(o,X_n) \textrm{ and } b_n:=\sum_{m\geq Kn^2+1}
p^{(m)}(o,X_n).
$$
By Lemma \ref{lemma:prob-bound},
$$
b_n\leq \sum_{m\geq Kn^2+1} \frac{G(o,o|\varrho)}{\varrho^{m}}\cdot  \Bigl(\frac{1}{1-\max_{i\in\calI} \xi_i(\varrho)}\Bigr)^n 
= G(o,o|\varrho)\cdot  \Bigl(\frac{1}{1-\max_{i\in\calI} \xi_i(\varrho)}\Bigr)^n \cdot
\frac{\varrho^{-Kn^2-1}}{1-\varrho^{-1}}. 
$$
Therefore, $b_n$ decays faster than any geometric sequence. Applying Lemma
\ref{cut-lemma} yields
$$
\ell=\lim_{n\to\infty}  -\frac{1}{n} \log \sum_{m=0}^{Kn^2}
p^{(m)}(o,X_n)\quad \textrm{almost surely.}
$$
By Lemma \ref{lemma:sum-bounds}, 
we may apply the Dominated Convergence Theorem and get:
\begin{eqnarray*}
\ell & = &  \int \lim_{n\to\infty} -\frac{1}{n} \log \sum_{m=0}^{Kn^2}
p^{(m)}(o,X_n)\, d\Prob \\
&=&
\lim_{n\to\infty} \int -\frac{1}{n} \log \sum_{m=0}^{Kn^2}
p^{(m)}(o,X_n)\, d\Prob\\
&=& \lim_{n\to\infty} -\frac{1}{n} \sum_{x\in V} p^{(n)}(o,x) \log \sum_{m=0}^{Kn^2}
p^{(m)}(o,x).
\end{eqnarray*}
Recall that \textit{Shannon's Inequality} gives
$$
-\sum_{x\in V} p^{(n)}(o,x) \log \mu(x) \geq -\sum_{x\in V} p^{(n)}(o,x) \log p^{(n)}(o,x)
$$
for every finitely supported probability measure $\mu$ on $V$. 
We apply now this inequality by setting $\mu(x):=\frac{1}{Kn^2+1}
\sum_{m=0}^{Kn^2} p^{(m)}(o,x)$:
\begin{eqnarray*}
\ell & \geq & \limsup_{n\to\infty} \frac{1}{n} \sum_{x\in V} p^{(n)}(o,x) \log
(Kn^2+1) -\frac{1}{n} \sum_{x\in V} p^{(n)}(o,x) \log
p^{(n)}(o,x) \\
&=& \limsup_{n\to\infty} -\frac{1}{n}\int \log \pi_n(X_n)\, d\Prob.
\end{eqnarray*}
Now we can conclude with Fatou's Lemma:
\begin{eqnarray}
\ell  \leq \int \liminf_{n\to\infty} \frac{-\log \pi_n(X_n)}{n} d\Prob & \leq &
\liminf_{n\to\infty} \int \frac{-\log \pi_n(X_n)}{n} d\Prob \nonumber\\
& \leq &
\limsup_{n\to\infty} \int \frac{-\log \pi_n(X_n)}{n} d\Prob \leq \ell.\label{eq:entropy-final}
\end{eqnarray}
Thus, $\lim_{n\to\infty} -\frac{1}{n} \E\bigl[\log \pi_n(X_n)\bigr]$ exists
and the limit equals $\ell$. The rest follows from (\ref{ell0-limit}) and (\ref{Ch-formel}).
\end{proof}
We now give another formula for the asymptotic entropy which shows that
it is strictly positive.
\begin{Th}\label{th:entropy2}
Assume $R>1$. Then the asymptotic entropy is given by
$$
h= \ell_0 \cdot \sum_{g,h\in V_{\ast}^\times} -\pi\bigl(g,\tau(g)\bigr)\,q\bigl((g,\tau(g)),(h,\tau(h))\bigr)\,\log
q\bigl((g,\tau(g)),(h,\tau(h))\bigr) > 0.
$$
\end{Th}
\textit{Remarks:} Observe that the sum on the right hand side of Theorem \ref{th:entropy2} equals the entropy rate (for positive recurrent Markov chains) of
$\bigl(\widetilde{\mathbf{W}}_k,\t_k\bigr)_{k\in\N}$, which is defined by the almost sure
constant limit
$$
h_Q:=\lim_{n\to\infty} -\frac{1}{n}\log
\mu_n\bigl((\widetilde{\mathbf{W}}_1,\tau_1),\dots,(\widetilde{\mathbf{W}}_n,\tau_n)\bigr),
$$ 
where
$\mu_n\bigl((g_1,\t_1),\dots,(g_n,\t_n)\bigr)$ is the joint distribution of
$\bigl((\widetilde{\mathbf{W}}_1,\t_1),\dots,(\widetilde{\mathbf{W}}_n,\t_n)\bigr)$. That is,
$h=\ell\cdot h_Q$. For more details, we refer e.g. to Cover and Thomas
\cite[Chapter 4]{cover-thomas}.
\par
At this point it is essential that we have defined the length function $l(\cdot)$
with the help of the functions $L(x,y|z)$ and not by the Greenian metric.

\begin{proof}
For a moment let be $x=x_1\dots x_n\in V$. Then:
\begin{eqnarray}
l(x)
& = & -\log \prod_{j=1}^{n} L_{\tau(x_j)}\bigl(o_{\tau(x_{j})}, x_j|\xi_{\tau(x_j)}\bigr)\nonumber
\\
&=&  
-\log \prod_{j=2}^{n} \frac{\alpha_{\tau(x_j)}}{\alpha_{\tau(x_{j-1})}}
\frac{\xi_{\tau(x_{j-1})}}{\xi_{\tau(x_j)}}
\frac{1-\xi_{\tau(x_{j})}}{1-\xi_{\tau(x_{j-1})}}
L_{\tau(x_j)}\bigl(o_{\tau(x_j)}, x_j|\xi_{\tau(x_j)}\bigr)\nonumber \\
&&\hspace{3cm}
-\log L_{\tau(x_1)}\bigl(o_{\tau(x_{1})}, x_1|\xi_{\tau(x_1)}\bigr)
+\log
\frac{\xi_{\tau(x_1)}\,\alpha_{\tau(x_n)}\,(1-\xi_{\tau(x_n)})}{\alpha_{\tau(x_1)}\, \xi_{\tau(x_n)}\,(1-\xi_{\tau(x_1)})}\nonumber\\
&=& 
-\log \prod_{j=2}^{n} q\bigl(( x_{j-1},\tau(x_{j-1})),(
x_j,\tau(x_j))\bigr) \nonumber\\
&& \hspace{3cm} -\log L_{\tau(x_1)}\bigl(o_{\tau(x_{1})},x_1|\xi_{\tau(x_1)}\bigr) + 
\log
\frac{\xi_{\tau(x_1)}\,\alpha_{\tau(x_n)}\,(1-\xi_{\tau(x_n)})}{\alpha_{\tau(x_1)}\, \xi_{\tau(x_n)}\,(1-\xi_{\tau(x_1)})}.\label{q-entropy}
\end{eqnarray}
We now replace $x$ by $X_{\e{k}}$ in the last equation: since $l(X_n)/n$ tends to $h$
almost surely, the subsequence $\bigl(l(X_{\e{k}})/\e{k}\bigr)_{k\in\N}$
converges also to $h$.
Since $\min_{i\in\calI} \xi_i >0$ and $\max_{i\in\calI} \xi_i <1$, we get 
$$
\frac{1}{\e{k}} \log \frac{\xi_{\tau(x_1)}\,\alpha_{\tau(x_k)}\,(1-\xi_{\tau(x_k)})}{\alpha_{\tau(x_1)}\, \xi_{\tau(x_k)}\,(1-\xi_{\tau(x_1)})}
\xrightarrow{k\to\infty} 0 \quad \textrm{ almost surely}, 
$$ 
where $x_1:=X_{\mathbf{e}_1}$ and $x_k:=\widetilde{\mathbf{W}}_k = \widetilde{X}_{\mathbf{e}_k}$.
By positive recurrence of $\bigl(\widetilde{\mathbf{W}}_k,\t_k\bigr)_{k\in\N}$, an
application of the ergodic theorem yields
\begin{eqnarray*}
&&-\frac{1}{k} \log \prod_{j=2}^{k}
q\bigl((\widetilde{\mathbf{W}}_{j-1},\t_{j-1}),(\widetilde{\mathbf{W}}_{j},\t_j)\bigr) \\
&\xrightarrow{n\to\infty}& 
h':=-\sum_{\substack{g,h\in V_{\ast}^\times;\\ \tau(g)\neq \tau(h)}}  \pi\bigl(g,\tau(g)\bigr)\,q\bigl((g,\tau(g)),(h,\tau(h)\bigr) \log q\bigl((g,\tau(g)),(h,\tau(h)\bigr)>0 \ \textrm{a.s.},
\end{eqnarray*}
whenever $h'<\infty$. Obviously, for every $x_1\in V_\ast^\times$
$$
\lim_{k\to\infty} -\frac{1}{\e{k}} \log
L_{\tau(x_1)}\bigl(o_{\tau(x_1)},x_1|\xi_{\tau(x_1)}\bigr)=0 \quad
\textrm{almost surely.}
$$
Since $\lim_{k\to\infty} k/\e{k}=\ell_0$  we get
$$
h = \lim_{k\to\infty} \frac{l\bigl(X_{\e{k}}\bigr)}{\e{k}} = 
\lim_{k\to\infty} \frac{l\bigl(X_{\e{k}}\bigr)}{k}\frac{k}{\e{k}} = h' \cdot \ell_0,
$$ 
whenever $h'<\infty$.
In particular, $h>0$ since $\ell_0>0$ by \cite[Section 4]{gilch:07}.
\par
It remains to show that
it cannot be that $h'=\infty$. For this purpose, assume now $h'=\infty$. Define
for $N\in\N$ the function 
$h_N: \bigl(V_\ast^\times\bigr)^2 \to \R$ by
$$
h_N(g,h):= N \land \bigl(-\log q\bigl((g,\tau(g)),(h,\tau(h))\bigr).
$$
Then 
\begin{eqnarray*}
&&-\frac{1}{k} \sum_{j=2}^k \log h_N  \bigl(\widetilde X_{\e{j-1}},\widetilde
X_{\e{j}}\bigr) \\
&\xrightarrow{k\to\infty}&
h_N':=-\sum_{\substack{g,h\in V_\ast^\times,\\ \tau(g)\neq \tau(h)}} \pi\bigl(g,\tau(g)\bigr)\,q\bigl((g,\tau(g)),(h,\tau(h))\bigr) \log h_N(g,h) \quad \textrm{ almost surely}.
\end{eqnarray*}
Observe that $h_N'\to\infty$ as $N\to\infty$.
Since $h_N(g,h)\leq -\log q\bigl((g,\tau(g)),(h,\tau(h))\bigr)$ and $h'=\infty$ by assumption there is for every
$M\in\R$ and almost every trajectory of $\bigl(\widetilde{\mathbf{W}}_k\bigr)_{k\in\N}$
an almost surely finite random time $\mathbf{T_q}\in\N$ such that for all $k\geq \mathbf{T_q}$
\begin{equation}\label{M-inequality}
-\frac{1}{k} \sum_{j=2}^k \log q\bigl((\widetilde{\mathbf{W}}_{j-1},\t_{j-1}),(\widetilde{\mathbf{W}}_j,\t_j)\bigr)
> M.
\end{equation}
On the other hand side there is for every $M>0$, every small $\eps >0$ and almost every
trajectory an almost surely finite random time $\mathbf{T_L}$ such that for all $k\geq \mathbf{T_L}$
\begin{eqnarray*}
&&-\frac{1}{\e{k}} \sum_{j=1}^k \log
L_{\tau(X_{\e{j}})}\bigl(o_{\tau(X_{\e{j}})},\widetilde X_{\e{j}}|\xi_{\tau(X_{\e{j}})}\bigr) \in (h-\eps,
h+\eps)\quad \textrm{ and}\\
&& -\frac{1}{\e{k}} \sum_{j=2}^k \log q\bigl((\widetilde
X_{\e{j-1}},\t_{j-1}),(\widetilde X_{\e{j}},\t_j)\bigr) \\
&=&
 -\frac{k}{\e{k}}\frac{1}{k} \sum_{j=2}^k \log q\bigl((\widetilde X_{\e{j-1}},\t_{j-1}),(\widetilde X_{\e{j}},\t_j)\bigr)
> \ell_0\cdot M.
\end{eqnarray*}
Furthermore, since $\min_{i\in\calI}\xi_i>0$ and $\max_{i\in\calI} \xi_i<1$
there is an almost surely finite random time $\mathbf{T}_\eps \geq
\mathbf{T_L}$ such that for all $k\geq \mathbf{T}_\eps$ and all
$x_1=X_{\mathbf{e}_1}$ and $x_k=\widetilde{X}_{\mathbf{e}_k}$
\begin{eqnarray*}
&&-\frac{1}{\e{k}}\log
\frac{\xi_{\tau(x_1)}\,\alpha_{\tau(x_k)}\,(1-\xi_{\tau(x_k)})}{\alpha_{\tau(x_1)}\,
  \xi_{\tau(x_k)}\,(1-\xi_{\tau(x_1)})} \in (-\eps,\eps)
 \quad \textrm{ and } \\
&&
\frac{1}{\e{k}} \log
L_{\tau(x_1)}\bigl(o_{\tau(x_1)},x_1|\xi_{\tau(x_1)}\bigr) \in (-\eps,\eps).
\end{eqnarray*}
Choose now $M>(h+3\eps)/\ell_0$. Then we get the desired
contradiction, when we substitute in equality (\ref{q-entropy}) the vertex $x$
by $X_{\e{k}}$ with $k\geq \mathbf{T}_\eps$, divide by $\e{k}$ on both sides and see that the
left side is in $(h-\eps,h+\eps)$ and the rightmost side is bigger than
$h+\eps$. This finishes the proof of Theorem \ref{th:entropy2}.
\end{proof}
\begin{Cor}\label{cor:entropy2}
Assume $R>1$. Then we have for almost every path of the random walk $(X_n)_{n\in\N_0}$
$$
h=\liminf_{n\to\infty} -\frac{\log \pi_n(X_n)}{n}.
$$
\end{Cor}
\begin{proof}
Recall Inequality (\ref{equ:liminf-h}). Integrating both sides of this
inequality yields together with the inequality chain (\ref{eq:entropy-final})
that
$$
\int \liminf_{n\to\infty}  -\frac{\log \pi_n(X_n)}{n} -h  \,d\Prob = 0,
$$
providing that $h=\liminf_{n\to\infty} -\frac{1}{n} \log \pi_n(X_n)$ for almost every realisation of the random walk.
\end{proof}
The following lemma gives some properties concerning general measure theory:
\begin{Lemma}
\label{lemma:measuretheory}
Let $(Z_n)_{n\in\N_0}$ be a sequence of non-negative random variables and
\mbox{$0<c\in\R$.} Suppose that $\liminf_{n\to\infty} Z_n \geq c$ almost surely and
$\lim_{n\to\infty} \E[Z_n] = c$. Then the following holds:
\begin{enumerate}
\item $Z_n \xrightarrow{\Prob} c$, that is, $Z_n$ converges in probability to $c$.
\item If $Z_n$ is uniformly bounded then $Z_n\xrightarrow{L_1} c$, that is, $\int \bigl| Z_n-c\bigr| d\Prob \to 0$ as $n\to\infty$.
\end{enumerate}
\end{Lemma}
\begin{proof}
First, we prove convergence in probability of $(Z_n)_{n\in\N_0}$. For every $\delta_1>0$, there is some index $N_{\delta_1}$ such that for all
$n\geq N_{\delta_1}$
$$
\int Z_n\,d\Prob \in (c-\delta_1,c+\delta_1).
$$
Furthermore, due to the above made assumptions on $(Z_n)_{n\in\mathbb{N}_0}$ there is for every $\delta_2>0$ some index $N_{\delta_2}$ such
that for all $n\geq N_{\delta_2}$
\begin{equation}\label{equ:stochconv1}
\Prob[Z_n>c-\delta_1] > 1-\delta_2.
\end{equation}
Since $c=\lim_{n\to\infty} \int Z_n\,d\Prob$ it must be that
for every arbitrary but fixed $\eps>0$, every $\delta_1<\eps$ and for all $n$ big enough
$$
\Prob\bigl[Z_n>c-\delta_1\bigr]\cdot (c-\delta_1) + 
\Prob\bigl[Z_n>c +\eps\bigr]\cdot (\eps+\delta_1)
\leq \int Z_n\,d\Prob 
\leq c+\delta_1,
$$
or equivalently,
$$
\Prob\bigl[Z_n>c +\eps\bigr] \leq
\frac{c+\delta_1 -\Prob\bigl[Z_n>c-\delta_1\bigr]\cdot (c-\delta_1)}{\eps+\delta_1}.
$$
Letting $\delta_2 \to 0$ we get
$$
\limsup_{n\to\infty} \Prob\bigl[Z_n>c +\eps\bigr] \leq \frac{2\delta_1}{\eps+\delta_1}.
$$
Since we can choose $\delta_1$ arbitrarily small we get
$$
\Prob\bigl[Z_n>c +\eps\bigr] \xrightarrow{n\to\infty} 0
\quad \textrm{ for all } \eps>0. 
$$
This yields convergence in probability of $Z_n$ to $c$.
\par
In order to prove the second part of the lemma we define for any small $\eps>0$ and
$n\in\N$ the events
$$
A_{n,\eps}:=\bigl[|Z_n-c|\leq \eps\bigr]
\ \textrm{ and } B_{n,\eps}:=\bigl[|Z_n-c|> \eps\bigr].
$$
For arbitrary but fixed $\eps>0$, convergence in probability of
$Z_n$ to $c$ gives an integer  $N_\eps\in\N$ such that $\Prob[B_{n,\eps}]<\eps$ for all $n\geq
N_\eps$. Since $0\leq Z_n \leq M$ is assumed to be uniformly bounded, we get for $n\geq N_\eps$:
$$
\int |Z_n -c|\, d\Prob 
 = \int_{A_{n,\eps}} |Z_n -c|\, d\Prob+
\int_{B_{n,\eps}} |Z_n -c|\, d\Prob 
\leq  \eps + \eps\, (M+c) \xrightarrow{\eps\to 0} 0. 
$$
Thus, we have proved the second part of the lemma. 
\end{proof}
We can apply the last lemma immediately to our setting:
\begin{Cor}\label{cor:entropy3}
Assume $R>1$. Then we have the following types of convergence: 
\begin{enumerate}
\item Convergence in probability:
$$
-\frac{1}{n}\log \pi_n(X_n) \xrightarrow{\Prob} h.
$$ 
\item Assume that there is $c_0>0$ such that $p(x,y)\geq c_0$ whenever
$p(x,y)>0$. Then:
$$
-\frac{1}{n}\log \pi_n(X_n) \xrightarrow{L_1} h.
$$
\end{enumerate}
\end{Cor}
\begin{proof}
Setting $Z_n=-\frac{1}{n}\log \pi_n(X_n)$ and applying Lemma
\ref{lemma:measuretheory} yields the claim. Note that the assumption
$p(x,y)\geq c_0$ yields $0\leq \frac{-\log \pi_n(X_n)}{n}\leq -\log c_0$.
\end{proof}
The assumption of the second part of the last corollary is obviously satisfied
if we consider free products of \textit{finite} graphs.
\par
The reasoning in our proofs for existence of the entropy and its different properties (in particular, the reasoning in Section \ref{subsec:entropy}) is very similar to
 the argumentation in \cite{gilch-mueller}. However, the structure of free products of graphs is more complicated than in the case of directed covers as
 considered in \cite{gilch-mueller}. We outline the main differences to the reasoning in the aforementionend article. First, in \cite{gilch-mueller} a very
 similar rate of escape (compare \mbox{\cite[Theorem 3.8]{gilch-mueller}} with Proposition \ref{prop:ell-existence}) is considered, which arises from a length
 function induced by last visit generating functions. While the proof of existence of the rate of escape in \cite{gilch-mueller} is easy to check, we have
 to make more effort in the case of free products, since $-\log L_i(o_i,x|1)$ is not necessarily bounded for $x\in V_i$. Additionally, one has to study the various ingridients of the
 proof more carefully, since non-trivial loops are possible in our setting in contrast to random walks on trees. Secondly, in \cite{gilch-mueller} the
 invariant measure $\pi\bigl(g,\tau(g)\bigr)$ of our proof collapses to $\nu\bigl(\tau(g)\bigr)$, that is, in \cite{gilch-mueller} one has to study the sequence
 $\bigl(\tau(\mathbf{W}_k)\bigr)_{k\in\N}$, while in our setting we have to study the more complex sequence $\bigl(\widetilde{\mathbf{W}}_k,\tau(\mathbf{W}_k)\bigr)_{k\in\N}$; compare \mbox{\cite[proof of Theorem 3.8]{gilch-mueller}} with Lemma \ref{lemma:tildeW-process} and Proposition \ref{prop:ell-existence}.

\section{A Formula via Double Generating Functions}

\label{sec:dgf}

In this section we derive another formula for the asymptotic entropy. The main
tool is the following theorem of Sawyer and Steger \cite[Theorem 2.2]{sawyer}:
\begin{Th}[Sawyer and Steger]
\label{sawyer}
Let $(Y_n)_{n\in\N_0}$ be a sequence of real-valued random variables such that, for some
$\delta>0$,
$$
\mathbb{E}\biggl( \sum_{n\geq 0} \exp(-rY_n-sn)\biggl) = \frac{C(r,s)}{g(r,s)}
\quad \textrm{ for } 0<r,s<\delta,
$$
where $C(r,s)$ and $g(r,s)$ are analytic for $|r|,|s|<\delta$
and \mbox{$C(0,0)\neq 0$}. Denote by $g'_r$ and $g'_s$ the partial derivatives of $g$
with respect to $r$ and $s$. Then
$$
\frac{Y_n}{n} \xrightarrow{n\to\infty} \frac{g'_r(0,0)}{g'_s(0,0)} \quad
\textrm{ almost surely.}
$$
\end{Th}
Setting $z=e^{-s}$ and $Y_n:=-\log L(o,X_n|1)$ the expectation in Theorem
\ref{sawyer} becomes
$$
\mathcal{E}(r,z) = \sum_{x\in V} \sum_{n\geq 0} p^{(n)}(o,x) L(o,x|1)^r z^n
= \sum_{x\in V} G(o,x|z) L(o,x|1)^r.
$$
We define for $i\in\calI$, $r,z\in\mathbb{C}$:
\begin{eqnarray*}
\mathcal{L}(r,z) & := & 1+\sum_{n\geq 1} \sum_{x_1\dots x_n\in V\setminus \{o\}}
\prod_{j=1}^n L_{\tau(x_j)}\bigl( o_{\tau(x_j)},x_j| \xi_{\tau(x_j)}(z)\bigr)
\cdot L_{\tau(x_j)}\bigl( o_{\tau(x_j)},x_j | \xi_{\tau(x_j)}\bigr)^r,\\
\mathcal{L}_i^+(r,z) & := & \sum_{x\in V_i^\times}
L_i\bigl(o_i,x|\xi_i(z)\bigr) L_i(o_i,x|\xi_i)^r.
\end{eqnarray*}
Finally, $\mathcal{L}_i(r,z)$ is defined by
$$
\mathcal{L}_i^+(r,z) \cdot 
 \biggl( 1+ \sum_{n\geq 2}
\sum_{\substack{x_2\dots x_n\in  V^\times\setminus\{o\},\\ \tau(x_2)\neq i}}
\prod_{j=2}^n L_{\tau(x_j)}\bigl( o_{\tau(x_j)},x_j| \xi_{\tau(x_j)}(z)\bigr)
\cdot L_{\tau(x_j)}\bigl( o_{\tau(x_j)},x_j | \xi_{\tau(x_j)}\bigr)^r\biggr).
$$
With these definitions we have $\mathcal{L}(r,z)=1+\sum_{i\in\calI}
\mathcal{L}_i(r,z)$ and $\mathcal{E}(r,z)=G(o,o|z) \cdot
\mathcal{L}(r,z)$. Simple computations analogously to \cite[Lemma 4.2,
Corollary 4.3]{gilch:07} yield
$$
\mathcal{E}(r,z) = \frac{G(o,o|z)}{1-\mathcal{L}^\ast(r,z)}, \textrm{ where }
\mathcal{L}^\ast(r,z) = \sum_{i\in\calI} \frac{\mathcal{L}_i^+(r,z)}{1+\mathcal{L}_i^+(r,z)}.
$$
We now define $C(r,z):=G(o,o|z)$ and $g(r,z):=1-\mathcal{L}^\ast(r,z)$ and
apply Theorem \ref{sawyer} by differentiating $g(r,z)$ and evaluating the
derivatives at $(0,1)$:
\begin{eqnarray*}
\frac{\partial g(r,z)}{\partial r}\biggl|_{r=0,z=1} & = &
-\sum_{i\in\calI}
\frac{\sum_{x\in V_i^\times} L_i(o_i,x|\xi_i) \cdot \log
  L_i(o_i,x|\xi_i)}{\bigl(1+\sum_{x\in V_i^\times}
  L_i(o_i,x|\xi_i)\bigr)^2}\\
&=& -\sum_{i\in\calI} G_i(o_i,o_i|\xi_i)\cdot (1-\xi_i)^2 \cdot \sum_{x\in V_i^\times} G_i(o_i,x|\xi_i) \log
L_i(o_i,x|\xi_i)\\
&=& -\sum_{i\in\calI} G_i(o_i,o_i|\xi_i)\cdot (1-\xi_i)^2 \cdot \\
&&\quad\quad \cdot \biggl(
\sum_{x\in V_i} G_i(o_i,x|\xi_i) \log G_i(o_i,x|\xi_i)-\frac{\log G_i(o_i,o_i|\xi_i)}{1-\xi_i}\biggr),\\
\frac{\partial g(r,z)}{\partial s}\biggl|_{r=0,s=0} & = &
\sum_{i\in\calI} \frac{\partial }{\partial z} \Bigl( 1- \bigl(1-\xi_i(z)\bigr)
G_i\bigl(o_i,o_i|\xi_i(z)\bigr)\Bigr)\Bigl|_{z=1}\\
&=&\sum_{i\in\calI} \xi_i'(1) \cdot \bigl( G_i(o_i,o_i|\xi_i) - (1-\xi_i) G_i'(o_i,o_i|\xi_i)\bigr).
\end{eqnarray*}
\begin{Cor}\label{Cor:entropy-dgf}
Assume $R>1$. Then the entropy can be rewritten as
$$
h= \frac{\frac{\partial g(r,z)}{\partial r}(0,1)}{\frac{\partial g(r,z)}{\partial s}(0,1)}.
$$
\end{Cor}
\begin{flushright}
$\Box$
\end{flushright}

\section{Entropy of Random Walks on Free Products of Groups}

\label{sec:groups}

In this section let each $V_i$ be a finitely generated group $\Gamma_i$ with
identity $e_i=o_i$. W.l.o.g. we assume that the $V_i$'s are pairwise
disjoint. The free product is again a group with concatenation (followed by
iterated cancellations and contractions) as group operation. We write $\Gamma_i^\times:=\Gamma_i\setminus\{e_i\}$.
Suppose we are given a probability measure $\mu_i$ on $\Gamma_i\setminus\{e_i\}$ for
every $i\in\calI$ governing a random walk on $\Gamma_i$, that is,
$p_i(x,y)=\mu_i(x^{-1}y)$ for all $x,y\in\Gamma_i$. Let $(\alpha_i)_{i\in\calI}$ be a family of strictly
positive real numbers with $\sum_{i\in\calI} \alpha_i=1$. Then the random walk
on the free product $\Gamma:=\Gamma_1 \ast \dots \ast \Gamma_r$ is defined by
the transition probabilities $p(x,y)=\mu(x^{-1}y)$, where
$$
\mu(w) = 
\begin{cases}
\alpha_i \mu_i(w), & \textrm{ if } w\in \Gamma_i^\times,\\
0, & \textrm{ otherwise.}
\end{cases}
$$
Analogously, $\mu^{(n)}$ denotes the $n$-th convolution power of $\mu$. The random walk on $\Gamma$ starting at the identity $e$ of $\Gamma$ is again
denoted by the sequence $(X_n)_{n\in\N_0}$. In particular, the radius of convergence of the associated Green function is strictly bigger than $1$; see \cite[Theorem 10.10, Corollary 12.5]{woess}.
In the case of free products of groups it is well-known that the entropy exists and can be written as
$$
h=\lim_{n\to\infty} \frac{-\log \pi_n(X_n)}{n}
= \lim_{n\to\infty} \frac{-\log F(e,X_n|1)}{n};
$$
see Derriennic \cite{derrienic}, Kaimanovich and Vershik \cite{kaimanovich-vershik} and Blach\`ere,
Ha\"issinsky and Mathieu \cite{blachere-haissinsky-mathieu}. For free products
of \textit{finite} groups, Mairesse and Math\'eus \cite{mairesse1} give an
explicit formula for $h$, which remains also valid for free products of
countable groups, but in the latter case one needs the solution of an infinite
system of polynomial equations. In the following we will derive another formula
for the entropy, which holds also for free products of \textit{infinite} groups.
\par
We set $l(g_1\dots g_n) := -\log F(e,g_1\dots g_n|1)$. Observe that transitivity yields $F(g,gh|1)= F(e,h|1)$. Thus,
$$
l(g_1\dots g_n) = 
-\log \prod_{j=1}^n F(g_1\dots g_{j-1},g_1\dots g_j|1)
= -\sum_{j=1}^n \log F(e,g_j|1).
$$
First, we rewrite the following expectations as
\begin{eqnarray*}
\E l(X_n) & = & \sum_{i\in\calI} \sum_{g\in\Gamma_i} \alpha_i\mu_i(g) \sum_{h\in\Gamma} l(h)\,\mu^{(n)}(h),\\
 \E l(X_{n+1}) & = & \sum_{i\in\calI} \sum_{g\in\Gamma_i} \alpha_i\mu_i(g) \sum_{h\in\Gamma} 
l(gh)\,\mu^{(n)}(h).
\end{eqnarray*}
Thus,
\begin{eqnarray}
\E l(X_{n+1}) - \E l(X_n)
&=& \sum_{i\in\calI} \sum_{g\in\Gamma_i} \alpha_i\mu_i(g) \int \bigl(l(gh)-l(h)\bigr)\,d\mu^{(n)}(h)\nonumber\\
&=& \sum_{i\in\calI} \sum_{g\in\Gamma_i} \alpha_i\mu_i(g) \int -\log \frac{F(e,gX_n|1)}{F(e,X_n|1)}d\mu^{(n)}.\label{eqn:expectations}
\end{eqnarray}
Recall that $\Vert X_n\Vert\to \infty$ almost surely. That is, $X_n$ converges
almost surely to a random infinite word $X_\infty$ of the form
$x_1x_2\ldots\in\Big(\bigcup_{i=1}^r\Gamma_i^\times\Bigr)^{\N}$, where two
consecutive letters are not from the same $\Gamma_i^\times$.
Denote by $X_\infty^{(1)}$ the first letter of $X_\infty$. Let be $g\in
\Gamma_i^\times$. For $n\geq \e{1}$, the integrand in $(\ref{eqn:expectations})$ is constant: if $\tau\bigl(X_\infty^{(1)}\bigr)\neq i$ then
$$
\log \frac{F(e,gX_n|1)}{F(e,X_n|1)} = \log F(e,g),
$$
and if $\tau\bigl(X_\infty^{(1)}\bigr)= i$ then
$$
\log \frac{F(e,gX_n|1)}{F(e,X_n|1)} = \log \frac{F\bigl(e,gX_\infty^{(1)}\bigl|1\bigr)}{F\bigl(e,X_\infty^{(1)}\bigl|1\bigr)}.
$$
By \cite[Section 5]{gilch:07}, for $i\in\mathcal{I}$ and $g\in\Gamma_i^\times$,
\begin{eqnarray*}
\varrho(i)& := & \Prob[X_\infty^{(1)}\in \Gamma_i] =
1-(1-\xi_i)\,G_i(o_i,o_i|\xi_i) \quad \textrm{ and}\\
\Prob[X_\infty^{(1)}=g] & = & F(o_i,g|\xi_i) \, (1-\xi_i)\,G_i(o_i,o_i|\xi_i)
= (1-\xi_i)\,G_i(o_i,g|\xi_i).
\end{eqnarray*}
Recall that $F(e,g|1) = F_i(o_i,g|\xi_i)$ for each $g\in\Gamma_i$. We
get:
\begin{Th}\label{thm:entropy-groups}
 Whenever $h_i:=-\sum_{g\in\Gamma_i} \mu_i(g) \log \mu_i(g)<\infty$ for all
 $i\in\calI$, that is, when all random walks on the factors $\Gamma_i$ have
 finite single-step entropy, then the asymptotic entropy $h$ of the random walk on $\Gamma$ is given by
$$
h =-\sum_{i\in\calI}\sum_{g\in\Gamma_i} \alpha_i \mu_i(g)
\Bigl[\bigl(1-\varrho(i)\bigr) \, \log F_i(o_i,g|\xi_i) +
(1-\xi_i)\, G_i(o_i,o_i|\xi_i)\, \mathcal{F}(g)\Bigr],
$$
where
\begin{equation}\label{equ-mathcalF}
\mathcal{F}(g):=
\sum_{g'\in\Gamma_i^\times} F_i(o_i,g'|\xi_i) \log
\frac{F_i(o_i,gg'|\xi_i)}{F_i(o_i,g'|\xi_i)}\ \textrm{ for } g\in\Gamma_i.
\end{equation} 
\end{Th}
\begin{proof}
Consider the sequence $\E l(X_{n+1}) - \E l(X_n)$. If this sequence converges,
its limit must equal $h$. By the above made
considerations we get
\begin{eqnarray*}
&&\E l(X_{n+1}) - \E l(X_n)\\
&\xrightarrow{n\to\infty} &-\sum_{i\in\calI}\sum_{g\in \Gamma_i}  \mu(g) \biggl[(1-\varrho(i)) \log F_i(o_i,g|\xi_i) +\sum_{g'\in\Gamma_i^\times} \Prob[X_\infty^{(1)}=g'] \log \frac{F_i(o_i,gg'|\xi_i)}{F_i(o_i,g'|\xi_i)}\biggr],
\end{eqnarray*}
if the sum on the right hand side is finite.
We have now established the proposed formula, but it remains to verify
finiteness of the sum above. This follows from the following observations:
\par
\underline{Claim A:} $-\sum_{i\in\calI}\sum_{g\in \Gamma_i} \alpha_i \mu_i(g)
(1-\varrho(i)) \log F_i(o_i,g|\xi_i)$ is finite.
\par
Observe that $F_i(o_i,g|\xi_i)\geq \mu_i(g)\xi_i$ for $g\in\mathrm{supp}(\mu_i)$. Thus,
$$
0 < -\sum_{g\in \Gamma_i} \mu_i(g) \log F_i(o_i,g|\xi_i) \leq -\sum_{g\in
  \Gamma_i} \mu_i(g) \log \bigl(\mu_i(g)\,\xi_i\bigr) = h_i-\log \xi_i.
$$
This proves Claim A.
\par
\underline{Claim B:} $\sum_{i\in\calI}\sum_{g\in \Gamma_i} \alpha_i \mu_i(g)
(1-\xi_i) \sum_{g'\in\Gamma_i^\times} G_i(o_i,g'|\xi_i) \Bigl|\log
\frac{F_i(o_i,gg'|\xi_i)}{F_i(o_i,g'|\xi_i)}\Bigr|$ is finite.
\par
Observe that
  $F_i(o_i,gg'|\xi_i)/F_i(o_i,g'|\xi_i)=G_i(o_i,gg'|\xi_i)/G_i(o_i,g'|\xi_i)$. Obviously,
$$
\mu_i^{(n)}(g')\,\xi_i^n \leq G_i(o_i,g'|\xi_i)\leq \frac{1}{1-\xi_i} \quad
\textrm{ for every } n\in\N \textrm{ and } g'\in\Gamma_i.
$$
For $g\in \Gamma$ set $N(g):=\bigl\lbrace n\in\N_0 \,\bigl|\,
\mu^{(n)}(g)>0\bigr\rbrace$. Then: 
\begin{eqnarray*}
0 & < & \sum_{g'\in\Gamma_i^\times} \Prob[X_\infty^{(1)}=g']\cdot \big| \log
G_i(o_i,g'|\xi_i)\bigr| \\
&=& \sum_{g'\in\Gamma_i^\times} (1-\xi_i)\cdot G_i(o_i,g'|\xi_i)\cdot \big| \log G_i(o_i,g'|\xi_i)\bigr|\\
&=& \sum_{g'\in\Gamma_i^\times} (1-\xi_i)\cdot \sum_{n\in N(g')}
\mu_i^{(n)}(g')\cdot \xi_i^n\cdot \big| \log G_i(o_i,g'|\xi_i)\bigr|\\
&\leq &\sum_{g'\in\Gamma_i^\times} (1-\xi_i)\cdot \sum_{n\in N(g')}
\mu_i^{(n)}(g')\cdot \xi_i^n\cdot 
\max\bigl\lbrace -\log \bigl(\mu_i^{(n)}(g')\,\xi_i^n\bigr), -\log(1-\xi_i)\bigr\rbrace\\
&\leq &(1-\xi_i) \cdot \sum_{n\in N(g')} n\,\xi_i^n \cdot\underbrace{\frac{-1}{n}
\sum_{g'\in\Gamma_i} \mu_i^{(n)}(g') \log \mu_i^{(n)}(g')}_{(\ast)}
- (1-\xi_i)\,\log \xi_i \sum_{n\geq 1} n\,\xi_i^n\\
&&\quad - (1-\xi_i)\,\log (1-\xi_i) \sum_{n\geq 1} \xi_i^n.
\end{eqnarray*}
Recall that $h_i<\infty$ together with Kingman's subadditive ergodic theorem
implies existence of a constant $H_i\geq 0$  with 
\begin{equation}\label{convergence-hi}
\lim_{n\to\infty} -\frac{1}{n}\sum_{g\in\Gamma_i} \mu_i^{(n)}(g) \log \mu_i^{(n)}(g) = H_i.
\end{equation}
Thus, if $n\in\N$ is large enough, the sum $(\ast)$ is in the interval
$(H_i-\eps,H_i+\eps)$ for any arbitrarily small $\eps>0$. That is, the sum
$(\ast)$ is uniformly bounded for all $n\in\N$. From this follows that the rightmost side of the last inequality chain is finite.
\par
Furthermore, we have for each $g\in\Gamma_i$ with $\mu_i^{(n)}(g)>0$:
\begin{eqnarray*}
0  & < & \sum_{g'\in\Gamma_i^\times}\Prob[X_\infty^{(1)}=g']\cdot \big| \log G_i(o_i,gg'|\xi_i)\bigr|\\
&=&\sum_{g'\in\Gamma_i^\times} (1-\xi_i)\cdot G_i(o_i,g'|\xi_i)\cdot \bigl| \log G_i(o_i,gg'|\xi_i)\bigr|\\
&=& \sum_{g'\in\Gamma_i^\times} (1-\xi_i) \cdot\sum_{n\in N(g')}
\mu_i^{(n)}(g')\cdot \xi_i^n\cdot \bigl| \log G_i(o_i,gg'|\xi_i)\bigr|\\
&\leq & \sum_{g'\in\Gamma_i^\times} (1-\xi_i)\cdot \sum_{n\in N(g')}
\mu_i^{(n)}(g')\cdot \xi_i^n \cdot \max\bigl\lbrace -\log \bigl(\mu_i(g)\,\mu_i^{(n)}(g')\,\xi_i^{n+1}\bigr),-\log(1-\xi_i)\bigr\rbrace\\
&\leq & -(1-\xi_i) \cdot \sum_{n\in N(g')} \xi_i^n
\cdot \sum_{g'\in\Gamma_i} \mu_i^{(n)}(g') \log \mu_i^{(n)}(g') -
(1-\xi_i)\cdot \log \xi_i \cdot
\sum_{n\geq 1} (n+1)\,\xi_i^n\\
&&\quad - \log \mu_i(g)- \log (1-\xi_i).
\end{eqnarray*}
If we sum up over all $g$ with $\mu(g)>0$, we get:
\begin{eqnarray*}
&& -\underbrace{\sum_{i\in\calI}\sum_{g\in\Gamma_i} \alpha_i \mu_i(g)
(1-\xi_i) \sum_{n\in N(g')} \xi_i^n \sum_{g'\in\Gamma_i} \mu_i^{(n)}(g') \log \mu_i^{(n)}(g')}_{(I)} \\
&& \quad - \underbrace{\sum_{i\in\calI}\sum_{g\in\Gamma_i} \alpha_i
  \mu_i(g) (1-\xi_i)\,\log \xi_i\sum_{n\geq 1}
  (n+1)\,\xi_i^n }_{(II)}\\
&&\quad 
-\underbrace{\sum_{i\in\calI} \alpha_i \sum_{g\in\Gamma_i} \mu_i(g)\log
  \mu_i(g)}_{(III)}
- \underbrace{\sum_{i\in\calI} \alpha_i  \log(1-\xi_i)}_{<\infty}.
\end{eqnarray*}
Convergence of $(I)$ follows from (\ref{convergence-hi}), $(II)$ converges
since $\xi_i<1$ and $(III)$ is convergent by assumption $h_i<\infty$. This
finishes the proof of Claim B, and thus the proof of the theorem.
\end{proof}
Erschler and Kaimanovich \cite{erschler-kaimanovich:entropy-continuity} asked if drift and entropy of random walks on groups depend continuously on the probability measure, which governs the random walk. Ledrappier \cite{ledrappier:10} proves in his recent, simultaneous paper that drift and entropy of finite-range random walks on free groups vary analytically with the probability measure of constant support. By Theorem \ref{thm:entropy-groups}, we are even able to show continuity for free products of finitely generated groups, but restricted to nearest neighbour random walks with fixed set of generators.
\begin{Cor}\label{cor:analyticity}
Let $\Gamma_i$ be generated as a semigroup by $S_i$. Denote by $\mathcal{P}_i$ the set of probability measures $\mu_i$ on $S_i$ with $\mu_i(x_i)>0$ for all $x_i\in S_i$. Furthermore, we write $\mathcal{A}:=\{(\alpha_1,\dots,\alpha_r) \mid \alpha_i>0, \sum_{i\in\calI}\alpha_i=1\}$. Then the entropy function 
$$
h: \mathcal{A}\times\mathcal{P}_1 \times\dots \times \mathcal{P}_r \to \mathbb{R}: (\alpha_1,\dots,\alpha_r,\mu_1,\dots,\mu_r) \mapsto h(\alpha_1,\dots,\alpha_r,\mu_1,\dots,\mu_r)
$$
is real-analytic.
\end{Cor}
\begin{proof}
The claim follows directly with the formula given in Theorem \ref{thm:entropy-groups}: the involved generating functions $F_i(o_i,g|z)$ and $G_i(o_i,o_i|z)$ are analytic when varying the probability measure of constant support, and the values $\xi_i$ can also be rewritten as
$$
\xi_i=\sum_{k_1,\dots,k_r,l_{1,1},\dots,l_{r,|S_r|}\geq 1} x(k_1,\dots,k_r,l_{1,1},\dots,l_{r,|S_r|})\prod_{i\in\calI} \alpha_i^{k_i} \prod_{j=1}^{| S_i|} \mu_i(x_{i,j})^{l_{i,j}},
$$
where $S_i=\{x_{i,1},\dots,x_{i,|S_i|}\}$. This yields the claim.
\end{proof}
\textit{Remarks:}
\begin{enumerate}
\item Corollary \ref{cor:analyticity} holds also for the case of free products of \textit{finite} graphs, if one varies the transition probabilities continously  under the assumption that the sets $\{(x_i,y_i)\in V_i\times V_i\mid p_i(x_i,y_i)>0\}$ remain constant: one has to rewrite $\xi_i$ as power series in terms of (finitely many) $p_i(x_i,y_i)$ and gets analyticity with the formula given in Theorem \ref{thm:entropy}.
\item Analyticity holds also for the drift (w.r.t. the block length and w.r.t. the natural graph metric) of nearest neighbour random walks due to the formulas given \mbox{in \cite[Section 5 and 7]{gilch:07}.}
\item The formula for entropy and drift given in Mairesse and Math\'eus \cite{mairesse1} for random walks on free products of \textit{finite} groups depends also analytically on the transition probabilities.
\end{enumerate}

\section{Entropy Inequalities}

\label{sec:entropy-inequalities}

In this section we consider the case of free products of \textit{finite} sets
$V_1,\dots,V_r$, where $V_i$ has $|V_i|$ vertices. We want to
establish a connection between asymptotic entropy, rate of escape and the
volume growth rate of the free product $V$. For $n\in\N_0$, let $S_0(n)$ be the set
of all words of $V$ of block length $n$. The following lemmas give an
answer how fast the free product grows.
\begin{Lemma}
The sphere growth rate w.r.t. the block length is given by 
$$
s_0 := \lim_{n\to\infty} \frac{\log |S_0(n)|}{n}=\log \lambda_0,
$$
where $\lambda_0$ is the Perron-Frobenius
eigenvalue of the $r\times r$-matrix $D=(d_{i,j})_{1\leq,i,j\leq r}$ with
$d_{i,j}=0$ for $i=j$ and $d_{i,j}=|V_j|-1$ otherwise.
\end{Lemma}
\begin{proof}
Denote by $\widehat D$ the $r\times
r$-diagonal matrix, which has entries $|V_1|-1,|V_2|-1, \dots,
|V_r|-1$ on its diagonal. Let $\mathds{1}$ be the $(r\times 1)$-vector with all entries
equal to $1$. Thus, we can write
$$
|S_0(n)| = \mathds{1}^T \widehat D D^{n-1}\mathds{1}.
$$
Let $0<v_1\leq \mathds{1}$ and $v_2\geq
\mathds{1}$ be eigenvectors of $D$ w.r.t. the Perron-Frobenius \mbox{eigenvalue $\lambda_0$.} Then
\begin{eqnarray*}
|S_0(n)| & \geq &  \mathds{1}^T \widehat D D^{n-1} v_1  = C_1 \cdot \lambda_0^{n-1},\\
|S_0(n)| & \leq &  \mathds{1}^T \widehat D D^{n-1} v_2  = C_2 \cdot \lambda_0^{n-1},
\end{eqnarray*}
where $C_1,C_2$ are some constants independent from $n$. Thus,
$$
\frac{\log |S_0(n)|}{n} = \log |S_0(n)|^{1/n} \xrightarrow{n\to\infty} \log \lambda_0.$$
\end{proof}
Recall from the Perron-Frobenius theorem that $\lambda_0\geq \sum_{i=1,i\neq
  j}^r (|\Gamma_i|-1)$ for each $j\in\mathcal{I}$; in particular, $\lambda_0\geq
1$. We also take a look on the natural graph metric and its growth rate. For
this purpose, we define
$$
S_1(n):=\bigl\lbrace
x\in V \,\bigl|\, p^{(n)}(o,x)>0, \forall m<n: p^{(m)}(o,x)=0
\bigr\rbrace,
$$
that is, the set of
all vertices in $V$ which are at distance $n$ to the root $o$ w.r.t. the
natural graph metric.
\par
We now construct a new graph, whose adjacency matrix allows us to describe the exponential growth of $S_1(n)$ as $n\to\infty$.
For this purpose, we visualize the sets $V_1,\dots,V_r$ as graphs $\mathcal{X}_1,\dots,\mathcal{X}_r$ with vertex
sets $V_1,\dots,V_r$ equipped with the following edges: for $x,y\in V_i$, there
is a directed edge from $x$ to $y$ if and only if $p_i(x,y)>0$.
Consider now directed spanning trees $\mathcal{T}_1,\dots,\mathcal{T}_r$ of the
graphs $\mathcal{X}_1,\dots,\mathcal{X}_r$ such that the graph distances of vertices in
$\mathcal{T}_i$ to the root $o_i$ remain the same as in $\mathcal{X}_i$. We now
investigate the free product $\mathcal{T}=\mathcal{T}_1\ast \dots \ast
\mathcal{T}_r$, which is again a tree. We make the crucial observation that $\mathcal{T}$ can be seen
as the directed cover of a finite directed graph $F$, where $F$ is defined in
the following way:
\begin{enumerate}
\item The vertex set of $F$ is given by $\{o\} \cup \bigcup_{i\in\mathcal{I}}
  V_i^\times $ with root $o$.
\item The edges of $F$ are given as follows: first, we add all edges inherited
  from one of the trees $\mathcal{T}_1,\dots,\mathcal{T}_r$, where $o$ plays the
  role of $o_i$ for each $i\in\mathcal{I}$. Secondly, we add for all
  $i\in\mathcal{I}$ and every $x\in V_i^\times$ an edge from $x$ to each $y\in
  V_j^\times$, $j\neq i$, whenever there is an edge from $o_i$ to $y$ in $\mathcal{T}_j$.
\end{enumerate}
The tree $\mathcal{T}$ can be seen as a \textit{periodic tree}, which is also
called a \textit{tree with finitely many cone
  types}; for more details we refer to Lyons \cite{lyons:book} and Nagnibeda and
Woess \cite{woess2}. Now we are able to state the following lemma:
\begin{Lemma}
The sphere growth rate w.r.t. the natural graph metric defined by 
$$
s_1 := \lim_{n\to\infty} \frac{\log |S_1(n)|}{n}
$$
exists. Moreover, we have the equation $s_1=\log \lambda_1$, where $\lambda_1$ is the
Perron-Frobenius eigenvalue of the adjacency matrix of the graph $F$.
\end{Lemma}
\begin{proof}
Since the graph metric remains invariant under the
restriction of $V$ to $\mathcal{T}$ and since it is well-known that the
growth rate exists for periodic trees (see \mbox{Lyons \cite[Chapter
3.3]{lyons:book}),} we have existence of the limit $s_1$. More precisely, $|S_1(n)|^{1/n}$ tends to the Perron-Frobenius eigenvalue of the adjacency matrix of $F$ as $n\to\infty$. 
For sake of completeness, we remark that the root of
$\mathcal{T}$ plays a special role (as a cone type) but this does not affect the
application of the results about directed covers to our case.
\end{proof}
For $i\in\{0,1\}$, we write $B_i(n)=\bigcup_{k=0}^n S_i(k)$.
Now we can prove:
\begin{Lemma}
The volume growth w.r.t. the block length, w.r.t. the natural graph metric
respectively, is given by
$$
g_0:=\lim_{n\to\infty} \frac{\log|B_0(n)|}{n}=\log \lambda_0, \quad g_1:=\lim_{n\to\infty}
\frac{\log|B_1(n)|}{n}=\log \lambda_1 \quad \textrm{respectively}.
$$
\end{Lemma}
\begin{proof}
For ease of better readability, we omit the subindex $i\in\{0,1\}$ in the
following, since the proofs for $g_0$ and $g_1$ are completely analogous.
Choose any small $\varepsilon>0$. Then there is some $K_\varepsilon$ such that for all
$k\geq K_\varepsilon$
$$
\lambda^k e^{-k\varepsilon} \leq |S(k)|
\leq \lambda^k e^{k\varepsilon}.
$$
Write $C_\varepsilon=\sum_{i=0}^{K_\varepsilon-1} |S(i)|$. Then for $n\geq K_\varepsilon$:
\begin{eqnarray*}
|B(n)|^{1/n} & =& \sqrt[n]{\sum_{k=0}^n |S(k)|}
\leq \sqrt[n]{C_\varepsilon + \sum_{k=K_\varepsilon}^n \lambda^k e^{k\varepsilon}}
 =  \lambda e^{\varepsilon} \sqrt[n]{\frac{C_\varepsilon}{\lambda^n e^{n\varepsilon}} +
    \sum_{k=K_\varepsilon}^n \frac{1}{\lambda^{n-k} e^{(n-k)\varepsilon}}}\\
&\leq & \lambda e^{\varepsilon} \sqrt[n]{\frac{C_\varepsilon}{\lambda^n e^{n\varepsilon}} +
    (n-K_\varepsilon+1)} \xrightarrow{n\to\infty} \lambda e^{\varepsilon}.
\end{eqnarray*}
In the last inequality we used the fact $\lambda \geq 1$.
Since we can choose $\varepsilon>0$ arbitrarily small, we get
$\limsup_{n\to\infty}|B(n)|^{1/n}\leq \lambda$. Analogously:
$$
|B(n)|^{1/n}  
\geq  \sqrt[n]{C_\varepsilon + \sum_{k=K_\varepsilon}^n \lambda^k e^{-k\varepsilon}}
 =  \lambda \sqrt[n]{\frac{C_\varepsilon}{\lambda^n} +
    \sum_{k=K_\varepsilon}^n \frac{e^{-k\varepsilon}}{\lambda^{n-k}}}
\xrightarrow{n\to\infty} \lambda e^{-\varepsilon}.
$$
That is, $\lim_{n\to\infty}\frac{1}{n} \log |B(n)|=\log\lambda$.
\end{proof}
For $i\in\{0,1\}$, define $l_i:V\to\mathbb{N}_0$ by $l_0(x)=\Vert x\Vert$ and
$l_1(x)=\inf \{m\in\mathbb{N}_0 \mid p^{(m)}(o,x)>0\}$. Then the limits
$\ell_i=\lim_{n\to\infty} l_i(X_n)/n$ exist; see \cite[Theorem 3.3, Section
7.II]{gilch:07}.
Now we can establish a connection between entropy, rate of escape and volume growth:
\begin{Cor}\label{cor:inequalities}
$h \leq g_0 \cdot \ell_0$ and $h \leq g_1 \cdot \ell_1$.
\end{Cor}
\begin{proof}
Let be $i\in\{0,1\}$ and $\varepsilon >0$. Then there is some $N_\varepsilon\in\N$ such that for
all $n\geq N_\varepsilon$
$$
1-\varepsilon \leq 
\Prob\bigl( \bigl\lbrace
x\in V \mid -\log \pi_n(x) \geq (h-\varepsilon)n, l_i(x) \leq
(\ell_i+\varepsilon)n\bigr\rbrace\bigr) 
\leq
e^{-(h-\varepsilon)n} \cdot \bigl|B_i\bigl((\ell_i+\varepsilon)n\bigr)\bigr|.
$$
That is,
$$
(h-\varepsilon) + \frac{\log (1-\varepsilon)}{n} \leq (\ell_i+\varepsilon)
\cdot \frac{\log \bigl|B_i\bigl((\ell_i+\varepsilon)n\bigr)\bigr|}{(\ell_i+\varepsilon)n}.
$$
If we let $n$ tend to infinity and make $\varepsilon$ arbitrarily small, we get
the claim.
\end{proof}
Finally, we remark that an analogous inequality for random walks on groups was given by \mbox{Guivarc'h \cite{guivarch},} and more generally for space- and time-homogeneous Markov chains by Kai\-ma\-no\-vich and Woess \cite[Theorem 5.3]{kaimanovich-woess}.

\section{Examples}

\label{sec:examples}

\subsection{Free Product of Finite Graphs}

Consider the graphs $\mathcal{X}_1$ and $\mathcal{X}_2$ with the transition
probabilities sketched  in Figure \ref{entropy-bsp2}. We set
$\alpha_1=\alpha_2=1/2$. For the computation of $\ell_0$ we need the following
functions:
$$
\begin{array}{rclcrcl}
F_1(g_1,o_1|z)  & = &\frac{z^2}{2} \frac{1}{1-z^2/2},& \quad&
F_2(h_1,o_2|z)  &= & \frac{z^2}{2} \frac{1}{1-z^3/2},\\[2ex]
\xi_1(z) & =&
\frac{z/2}{1-\frac{z}{2}\frac{\xi_2(z)^2}{2}\frac{1}{1-\xi_2(z)^3/2}},& \quad&
\xi_2(z) &= & 
\frac{z/2}{1-\frac{z}{2}\frac{\xi_1(z)^2}{2}\frac{1}{1-\xi_1(z)^2/2}}.
\end{array}
$$
Simple computations with the help of \cite[Section 3]{gilch:07} and
\textsc{Mathematica} allow us to determine the rate of escape of the random
walk on $\mathcal{X}_1\ast \mathcal{X}_2$ as $\ell_0=0.41563$. For the computation of the entropy, we need also the following
generating functions:
\begin{eqnarray*}
&& L_1(o_1,g_1|z)  =  \frac{z}{1-z^2/2},\quad
L_1(o_1,g_2|z)  =  \frac{z^2}{1-z^2/2},\quad
L_2(o_2,h_1|z)  =  \frac{z}{1-z^3/2},\\
&& L_2(o_2,h_2|z)  =  \frac{z^2}{1-z^3/2},\quad
L_2(o_2,h_3|z)  =  \frac{z^3/2}{1-z^3/2}.
\end{eqnarray*}
Thus, we get the asymptotic entropy as $h=0.32005$.
\begin{figure}\label{entropy-bsp2}
\includegraphics[width=6cm]{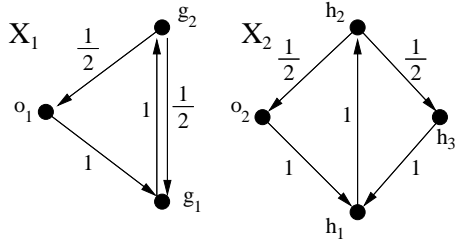}
\caption{Finite graphs $\mathcal{X}_1$ and $\mathcal{X}_2$}
\end{figure}

\subsection{$(\Z\times \Z/2)\ast  (\Z\times \Z/2)$}
Consider the free product $\Gamma=\Gamma_1\ast\Gamma_2$ of the \textit{infinite}
groups $\Gamma_i=\Z\times (\Z/2\Z)$ with $\alpha_i=1/2$ and $\mu_i\bigl((\pm
1,0)\bigr)= \mu_i\bigl((0,1)\bigr)=1/3$ for each $i\in\{1,2\}$. We set
$a:=(1,0)$, $b:=(1,1)$, $c:=(0,1)$ and $\lambda(x,y):=x$ for $(x,y)\in\Gamma_i$.
Define
\begin{eqnarray*}
\hat F(a|z) & := & \sum_{n\geq 1} \Prob\bigl[Y_n=a, \forall m<n: \lambda(Y_m)<1
\,\bigl|\, Y_0=(0,0)\bigr]\,z^n,\\
\hat F(b|z) & := & \sum_{n\geq 1} \Prob\bigl[Y_n=b, \forall m<n: \lambda(Y_m)<1\,\bigl|\, Y_0=(0,0)\bigr]\,z^n,
\end{eqnarray*}
where $(Y_n)_{n\in\mathbb{N}_0}$ is a random walk on $\Z\times \Z/2$ governed by $\mu_1$.
The above functions satisfy the following system of equations:
\begin{eqnarray*}
\hat F(a|z) & = & \frac{z}{3} \Bigl( 1+ \hat F(b|z) + \hat F(a|z)^2 + \hat
F(b|z)^2\Bigr),\\
\hat F(b|z) & = & \frac{z}{3} \Bigl(\hat F(a|z) + \hat F(a|z) \hat F(b|z) + \hat
F(b|z)\hat F(a|z)\Bigr).
\end{eqnarray*}
From this system we obtain explicit formulas for $\hat F(a|z)$ and $\hat
F(b|z)$. We write $F\bigl(n,j|z):=F_1\bigl((0,0),(n,j)|z)$ for $(n,j)\in
\Z\times \Z/2$.
To compute the entropy rate  we  have to solve the following
system of equations:
\begin{eqnarray*}
F(a|z) & = & \frac{z}{3} \Bigl(1+F(b|z) + \hat F(a|z) F(a|z)
+ \hat F(b|z) F(b|z)\Bigr),\\
F(b|z) & = & \frac{z}{3} \Bigl(F(c|z)+F(a|z) + \hat F(a|z)
F(b|z)+ \hat F(b|z) F(a|z)\Bigr),\\
F(c|z) & = & \frac{z}{3} \Bigl( 1+ 2\,F(b|z)\Bigr).
\end{eqnarray*}
Moreover, we need the value $\xi_1(1)=\xi_2(1)=\xi$. This value can be computed
 analogously to \cite[Section 6.2]{gilch:07}, that is, $\xi$ has to be
computed numerically from the equation
$$
\frac{\xi}{2-2\xi} = \xi\, G_1(\xi) = \frac{\xi}{1-\frac{2}{3}\xi F(a|\xi)
  -\frac{1}{3}\xi F(c|\xi)}.
$$
Solving this equation with \textsc{Mathematica} yields $\xi=0.55973$. To
compute the entropy we have to evaluate the functions $F(g|z)$ at $z=\xi$ for
each $g\in\Z\times \Z_2$. For even $n\in\N$, we have the following formulas:
\begin{eqnarray*}
F\bigl((\pm n,0)| \xi \bigr) 
& = & \sum_{k=0}^{n/2} {n \choose 2k} \hat
F(b|\xi)^{2k} \hat F(a|\xi)^{n-2k}+ \\
&&\quad \sum_{k=0}^{n/2-1} {n \choose 2k+1} \hat
F(b|\xi)^{2k+1} \hat F(a|\xi)^{n-2k-1} F(c|\xi),\\
F\bigl((\pm n,1)| \xi \bigr)
 & = & \sum_{k=0}^{n/2-1} {n \choose 2k+1} \hat
F(b|\xi)^{2k+1} \hat F(a|\xi)^{n-2k-1}+\\
&&\quad \sum_{k=0}^{n/2} {n \choose 2k} \hat
F(b|\xi)^{2k} \hat F(a|\xi)^{n-2k} F(c|\xi).
\end{eqnarray*}
For odd $n\in\N$,
\begin{eqnarray*}
F\bigl((\pm n,0)| \xi \bigr)
 & = & \sum_{k=0}^{(n-1)/2} {n \choose 2k} \hat
F(b|\xi)^{2k} \hat F(a|\xi)^{n-2k}+\\
&&\quad  \sum_{k=0}^{(n-1)/2} {n \choose 2k+1} \hat
F(b|\xi)^{2k+1} \hat F(a|\xi)^{n-2k-1} F(c|\xi),\\
F\bigl((\pm n,1)| \xi \bigr)
 & = & \sum_{k=0}^{(n-1)/2} {n \choose 2k+1} \hat
F(b|\xi)^{2k+1} \hat F(a|\xi)^{n-2k-1}+\\
&&\quad  \sum_{k=0}^{(n-1)/2} {n \choose 2k} \hat
F(b|\xi)^{2k} \hat F(a|\xi)^{n-2k} F(c|\xi).
\end{eqnarray*}
Moreover, we define $\hat F := \Prob[\exists n\in\N: \lambda(X_n)=1]$. This
probability can be computed by conditioning on the first step and solving
$$
\hat F = \frac{\xi}{3} \bigl( 1+ \hat F + \hat F^2\bigr),
$$
that is, $\hat F=0.24291$. Observe that we get the following estimations:
\begin{eqnarray*}
F_1(o,g|\xi) & \leq & \hat F^{|\lambda(g)|} \quad \textrm{ for } g\in \Z\times \Z_2,\\
F_1(o,g|\xi) & \geq & \hat F^{|\lambda(g)|-1}\cdot \min\bigl\lbrace
F_1(o_1,a|\xi), F_1(o_1,b|\xi) \bigr\rbrace \quad \textrm{ for } g\in \bigl(\Z\times \Z_2\bigr)\setminus \{(0,0),c\}.
\end{eqnarray*}
These bounds allow us to cap the sum over all $g'\in\Gamma_i^\times$ in
(\ref{equ-mathcalF}) and to estimate the tails of these sums. Thus, we can
compute the entropy rate numerically as $h=1.14985$.

\section*{Acknowledgements}

The author is grateful to Fr\'ed\'eric Math\'eus 
for discussion on the problems and several hints regarding content and exposition.


\end{document}